\documentclass[11pt]{amsart}
\usepackage{stmaryrd}
\usepackage{amssymb}
\usepackage{csvsimple}
\usepackage{siunitx}
\usepackage{hyperref}
\usepackage{url}
\usepackage{graphicx,color}
\usepackage{lineno}           
\usepackage[style=trad-abbrv,url=false,doi=true,
  eprint=true,isbn=false,backend=biber]{biblatex}

\def\R{\mathbb{R}}

\def\N{\mathbb{N}}

\def\Rinf{\R\cup \{+\infty\}}

\def\cD{\mathcal{D}}

\def\cI{\mathcal{I}}
\def\cJ{\mathcal{J}}

\def\cL{\mathcal{L}}

\def\cP{\mathcal{P}}

\def\cR{\mathcal{R}}
\def\cS{\mathcal{S}}
\def\cT{\mathcal{T}}

\def\a{\alpha}
\def\b{\beta}
\def\g{\gamma}
\def\d{\delta}

\def\s{\sigma}
\def\p{\partial}

\def\veps{\varepsilon}

\def\O{\Omega}

\def\GD{{\Gamma_D}}
\def\GN{{\Gamma_N}}

\def\transp{{\sf T}}

\def\hu{\widehat{u}}
\def\tu{\widetilde{u}}

\newcommand{\dv}[1]{\,{\mathrm d}#1}

\newcommand{\wcheck}[1]{#1\hspace{-.8ex}\mbox{\huge {\lower.45ex \hbox{$\textstyle \check{}$}}} \hspace{.5ex}}

\newcommand{\jump}[1]{\llbracket#1\rrbracket}   
\DeclareMathOperator{\id}{id}

\DeclareMathOperator{\diver}{div}

\DeclareMathOperator{\diam}{diam}

\DeclareMathOperator{\sign}{sign}

\let\oldmarginpar\marginpar
\renewcommand\marginpar[1]{
  \oldmarginpar[\raggedleft\footnotesize #1]
  {\raggedright\footnotesize #1}}

\newtheorem{definition}{Definition}

\newtheorem{proposition}[definition]{Proposition}

\newtheorem{corollary}[definition]{Corollary}
\newtheorem{remark}[definition]{Remark}

\newtheorem{example}[definition]{Example}

\numberwithin{definition}{section}
\definecolor{modmag}{RGB}{179,0,229}

\renewcommand{\text}{\textnormal}
\def\RT{{\mathcal{R}T}}
\def\CR{{cr}}

\def\tg{\widetilde{g}}
\def\tI{\widetilde{I}}

\def\tz{{\widetilde{z}}}

\def\hT{\widehat{T}}
\def\tD{\widetilde{D}}
\def\stop{{stop}}

\def\hz{\widehat{z}}
\def\tg{\widetilde{\g}}

\begin{document}
\title[Local mesh refinement for TV minimization]{Singular solutions, graded meshes,
and adaptivity for total-variation regularized minimization problems}
\author[S. Bartels]{S\"oren Bartels}
\address{Abteilung f\"ur Angewandte Mathematik,  
Albert-Ludwigs-Universit\"at Freiburg, Hermann-Herder-Str.~10, 
79104 Freiburg i.~Br., Germany}
\email{bartels@mathematik.uni-freiburg.de}
\author[R. Tovey]{Robert Tovey}
\address{INRIA de Paris, 2 Rue Simone IFF, 75012 Paris, France}
\email{robert.tovey@inria.fr}
\author[F. Wassmer]{Friedrich Wassmer}
\address{Abteilung f\"ur Angewandte Mathematik,  
Albert-Ludwigs-Universit\"at Freiburg, Hermann-Herder-Str.~10, 
79104 Freiburg i.~Br., Germany}
\email{friedrich.wassmer@pluto.uni-freiburg.de}
\date{\today}
\renewcommand{\subjclassname}{
\textup{2010} Mathematics Subject Classification}
\begin{abstract}
Recent quasi-optimal error estimates for the finite element approximation of 
total-variation regularized minimization problems require the existence of 
a Lipschitz continuous dual solution. We discuss the validity of this condition and
devise numerical methods using locally refined meshes that lead to 
improved convergence rates despite the occurrence of discontinuities. It turns out
that nearly linear convergence is possible on suitably constructed meshes.
\end{abstract}

\maketitle

\section{Introduction} 
In this article we consider the finite element discretization of the
Rudin--Osher--Fatemi (ROF) model from \cite{ROF}
which serves as a model problem for general convex and nonsmooth minimization problems. 
This widely used model in image processing determines a function $u \in BV(\O) \cap L^2(\O)$ 
via a minimization of 
\[
I(u) = |Du|(\O) + \frac{\alpha}{2} \|g-u\|^2,
\]
where $|Du|(\O)$ denotes the total variation of $u \in BV(\O)\cap L^2(\O)$, $g \in L^2(\O)$ 
is the input data, for example a noisy image, and $\|g-u\|$ is the $L^2$ distance between 
the given image and its regularization. The fidelity parameter $\alpha > 0$ is also given 
and determines the balance between denoising and preserving the input image. 
For more information on analytical features, explicit solutions in particular examples, 
and numerical methods concerning this model we refer the reader to 
\cite{ChaLio97,Attouch2006,Ambrosio2000,HinKun04,CCCNP10,ChLeLu11,WanLuc11,Bart12,LaiMat12,Burg16,BeEfRu17,HHSVW19,CaiCha20-pre,ChaPoc21}. 
Since this model allows for and 
preserves discontinuities of the input function $g$, cf. \cite{Jump-set-paper}, continuous 
finite element methods are known to perform suboptimally, cf. \cite{BaNoSa15,Bart20a}. 
Recent results in \cite{Chamb} and \cite{Bart20a,Bart20b-pre} show that quasi-optimal convergence 
rates~$O(h^{1/2})$ for discontinuous solutions on quasi-uniform triangulations 
can be obtained by using discontinuous, low order Crouzeix-Raviart finite elements from 
\cite{Crouzeix-Raviart} or appropriate discontinous Galerkin methods. 
The error estimates bound the error for approximating minimizers for $I$ by minimizing
the discrete functional
\[
I_h(u_h) = \int_\O |\nabla_h u_h| \dv{x} + \frac{\alpha}{2} \|\Pi_{h}(g-u_h)\|^2
\]
over piecewise affine functions $u_h \in \cS^{1,cr}(\cT_h)$ that are continuous at midpoints
of element sides. Here $\nabla_h$ denotes the elementwise gradient and
$\Pi_h$ is the projection onto elementwise constant functions on the triangulation $\cT_h$.
Note that the functional $I_h$ defines a nonconforming approximation of $I$, as, e.g., 
jump terms of $u_h$ across interelement sides are not included. 
The quasi-optimal rate applies if the dual problem, given by a maximization
of
\[
D(z) = -\frac{1}{2\a} \|\diver z + \a g \|^2 + 
\frac{\alpha}{2} \|g\|^2 - I_{K_1(0)}(z)
\]
over vector fields $z\in W^2_N(\diver;\O)$ admits a Lipschitz continuous solution.
The indicator functional $I_{K_1(0)}$ of the closed unit ball centered at the origin
enforces the pointwise constraint $|z|\le 1$. Although Lipschitz continuity
is known to be true in some settings, following
an idea from~\cite{tovey2020} we show that the condition is not satisfied in general and 
lower convergence rates have to be expected. The generic counterexample uses the 
difference of two characteristic functions of balls $B_r^\pm = B_r(\pm r,0)$ with radius
$r>0$ that touch in the origin, i.e., 
\[
g = \chi_{B_r^+} - \chi_{B_r^-}.
\]
Precise characterizations, cf.~\cite{Jump-set-paper},
of dual solutions along the jump set of the primal solution,
given by 
\[
u = c_{r,\a} g,  \quad c_{r,\a} = \max\{1-d/(\alpha r),0\},
\]
imply that Lipschitz continuity of dual solutions fails at the origin if $\a r > d$. 
Surprisingly, this singularity does not appear to affect the convergence rates of approximations 
on sequences of uniform triangulations. 

To obtain the quasi-optimal rate $O(h)$ for weakly differentiable functions 
we investigate two numerical methods that construct 
locally refined meshes. The first approach uses the fact that the jump set~$J_u$ of the 
primal solution~$u$ is contained in the jump set~$J_g$ of the given function $g$, i.e., 
\[
J_u\subset J_g
\]
cf.~\cite{Jump-set-paper}. Reduced convergence rates are related to the 
suboptimal approximation of jumps and therefore our idea is to refine triangulations in 
a neighborhood of the jump set $J_g$. Since we aim at preserving shape regularity,
the grading strength is limited and it turns out that the minimal mesh-size $h_{\min}$ 
used at the discontinuity set cannot be smaller than $h^\b$ with $\b \le 2$ and
the average mesh-size~$h$. Our numerical analysis shows that a quadratic grading
is under suitable conditions on a piecewise regular solution, the correct refinement
strength to obtain a nearly linear convergence rate, i.e., we have 
\[
\|u-u_h\| = O(h |\log h|),
\]
where $u_h$ is the Crouzeix--Raviart finite element solution. In one-dimensional situations
it coincides with the $P1$ approximation and the logarithmic factor can be omitted.

The approach of using graded meshes can only be efficiently applied if the given
function~$g$ is piecewise regular and the jump set $J_g$ is sufficiently simple. 
A more general concept uses refinement indicators based on a~posteriori error 
estimates that bound the approximation error of an approximation
$u_h$, e.g., a continuous $P1$ approximation, by computable quantities. These
depend on an approximation $u_h$ of the discrete primal problem and an admissible 
vector field $q$ for the continuous dual problem, i.e., we have 
\[
\frac{\a}{2} \|u- u_h \|^2 \le I(u_h) - I(u) \le I(u_h) - D(q) = \eta_h^2(u_h,q),
\]
which follows from coercivity properties of the functional $I$ and the 
duality principle $I(u) \ge D(q)$ for every admissible vector field~$q$. 
A simple calculation shows that if $u_h$ is weakly differentiable and
$q$ satisfies $|q|\le 1$ in $\O$, the error estimator is given 
as a sum of local, nonnegative quantities, i.e., 
\[
\eta_h^2(u_h,q) = \int_\O |\nabla u_h| - \nabla u_h \cdot q \dv{x}
+ \frac{1}{2\a} \int_\O \big(\diver q - \a (u_h -g) \big)^2 \dv{x}.
\]
By using the partitioning of the domain $\O$ given by the triangulation $\cT_h$
we obtain refinement indicators $\eta_T^2(u_h,q)$ that are used to refine
elements $T\in \cT_h$. The reliable error estimator can only be efficient if
the vector field $q$ is a nearly optimal approximation of a dual solution~$z$.
To avoid the expensive solution of a discretization of the dual problem
we use the observation that an approximation can be obtained via post-processing
the Crouzeix--Raviart approximation, cf.~\cite{Bart20a}. In particular,
this provides a maximizing vector field for the discrete dual functional 
\[
D_h(z_h) = -\frac{1}{2\a} \| \diver z_h +  \a \Pi_{h}g \|^2 
+ \frac{\alpha}{2} \|\Pi_{h}g \|^2 - I_{K_1(0)}(\Pi_h z_h),
\]
defined on Raviart--Thomas vector fields $z_h \in \RT^0_N(\cT_h)$. For this
definition we have the discrete duality principle $I_h(u_h) \ge D_h(z_h)$.
However, the unit length constraint $|z|\le 1$ is only imposed at midpoints
so that the optimal $z_h$ is in general inadmissible in the continuous dual
functional $D$. Since only the midpoint values of $z_h$ and its elementwise 
constant divergence enter the error estimator it nevertheless appears to be a 
reasonable way to define error indicators $(\eta_T)_{T\in \cT_h}$ although
the error bound may fail to hold; obvious corrections of $z_h$ do not seem
to lead to efficient error estimators. Our numerical experiments based on
a related adaptive mesh refinement algorithm lead to improved experimental
convergence rates that are lower than the ones obtained for graded meshes. Our
explanation for this is that the graded meshes are optimal
for the $L^2$ error while the coercivity estimate leading to the 
a~posteriori error estimate controls a stronger error quantity. 

This article is organized as follows. In Section~\ref{sec:fe_spaces} we 
introduce the used notation and define the relevant finite element spaces. 
The example with a non-Lipschitz dual solution for the ROF model is investigated 
in Section~\ref{sec:sing_sol}. In Section~\ref{sec:graded} the graded grid approaches are 
devised and analyzed. In Section~\ref{sec:prim-dual} the primal-dual error estimator is
defined and the construction of an optimal discrete
dual vector field via discrete duality relations is shown.
Numerical experiments are presented in Section~\ref{sec:num_ex}.


\section{Notation and finite element spaces}\label{sec:fe_spaces}
Given a bounded Lipschitz domain $\O\subset \R^d$
we use standard notation for Sobolev spaces $W^{s,p}(\O)$ and abbreviate
the norm in $L^2(\O)$ by
\[
\|\cdot \| = \| \cdot\|_{L^2(\O)}.
\]
The space of functions of bounded variation $BV(\O)$ consists of all
$u\in L^1(\O)$ such that its total variation
\[
|Du|(\O) = \sup_{\phi \in C^\infty_c(\O;\R^d),\, |\phi|\le 1} - \int_\O u \diver \phi \dv{x}
\]
is bounded. We refer the reader to~\cite{Ambrosio2000,Attouch2006} 
for properties of the space and note here that distributional gradients
$Du$ of functions $u\in BV(\O)$  can be decomposed into a regular, a jump,
and a Cantor part via 
\[
Du = \nabla u \otimes \dv{x} - \jump{u n} \otimes \dv{s}|_{J_u} + C_u.
\]
Vector fields $w \in L^q(\O;\R^d)$ that have a weak divergence 
$\diver w \in L^q(\O)$ and whose normal component $w\cdot n$ 
vanishes on the boundary part $\GN = \GD \setminus \p\O$  are contained
in the set $W^q_N(\diver;\O)$, we omit the subindex if $\GN= \emptyset$.

In the following $(\cT_h)_{h>0}$ denotes a sequence of regular, i.e., uniformly
shape regular and conforming, triangulations of the bounded 
Lipschitz domain $\O \subset \R^d$. The set $\cS_h$ contains the sides 
of elements. The parameter $h$ refers to an average 
mesh-size $h \sim (|\O|/N)^{1/d}$, where $N$ is the number of vertices of $\cT_h$. 
We furthermore let $h_T = \diam(T)$ for $T\in \cT_h$ and 
\[
h_{\max} = \max_{T\in \cT_h} h_T, \quad h_{\min} = \min_{T\in \cT_h} h_T.
\]
On quasi-uniform triangulations all mesh-sizes are comparable, i.e., 
$h_{\min} \sim h_{\max} \sim h$. We let $\cP_k (T) $ 
denote the set of polynomials of maximal total degree $k$ on $T \in \cT_h$ 
and define the set
of discontinuous, elementwise polynomial functions or vector fields as
\begin{align*}
\mathcal{L}^k(\cT_h)^\ell= \lbrace w_h \in L^{\infty}(\O,\R^\ell) : 
\ w_h|_T \in P^k(T)^\ell \ \text{for all} \ T \in \cT_h \rbrace.
\end{align*}
Barycenters of elements and sides are denoted by $x_T$ for all $T\in \cT_h$
and $x_S$ for all $S\in \cS_h$. 
The $L^2$ projection onto piecewise constant functions or vector fields
is denoted by 
\begin{align*}
\Pi_{h}: L^1(\O,\R^\ell) \to \mathcal{L}^0(\cT_h)^\ell.
\end{align*}
Note that we have $v_h(x_T) = \Pi_h v_h|_T$ for all $T\in \cT_h$ and
$v_h\in \cL^1(\cT_h)$. 

We next collect some elementary properties of standard and nonstandard
finite elements and refer the reader 
to~\cite{Crouzeix-Raviart,Raviart-Thomas,40years-CR,BreSco08-book,Bartels2016}
for further details. The $P1$-finite element space is defined via
\[
\cS^1(\cT_h) = \big\{ v_h \in \cL^1(\cT_h): \ v_h \ \text{continuous in } \overline{\O} \big\}.
\]
A low order Crouzeix-Raviart finite element space is given by
\begin{align*}
\cS^{1,cr}(\cT_h) = \big\{ v_h \in \cL^1(\cT_h): 
\ v_h \ \text{continuous in} \ x_S \ \text{for all} \ S \in S_h \big\}.
\end{align*}
We let $\cS^1_D(\cT_h)$ and $\cS^{1,cr}_D(\cT_h)$ denote the subspaces of functions
satisfying boundary conditions in vertices or in barycenters of sides on $\GD$,
respectively. The elementwise gradient $\nabla_h v_h\in \cL^0(\cT_h)^d$ of a
function $v_h \in \cS^{1,cr}(\cT_h)$ is defined via   
\begin{align*}
(\nabla_h v_h)|_{T} = \nabla (v_h|_T)
\end{align*}
for all $T \in \cT_h$. A low order Raviart-Thomas-finite element space 
is given by
\[
\cR T^0(\cT_h) = \lbrace y_h \in W^1(\diver;\O) \  : \ y_h|_T = a_T + b_T  (x - x_T), 
\ \text{for all} \ T \in \cT_h \rbrace.
\]
Vector fields $q_h \in \cR T^0(\cT_h)$ belong to $W^1(\diver;\O)$ and
have continuous constant normal components 
on sides of elements, we set $\RT^0_N(\cT_h) = \RT^0(\cT_h)\cap W^1_N(\diver;\O)$.
The spaces $\cS^{1,cr}_D(\cT_h)$ and $\RT^0_N(\cT_h)$ are connected via
the integration-by-parts formula
\[
\int_\O v_h \diver q_h \dv{x} = -\int_\O \nabla_h v_h \cdot q_h \dv{x} 
\]
for all $v_h \in \cS^{1,cr}_D(\cT_h)$ and $q_h \in \RT^0_N(\cT_h)$.

Given a function $v\in BV(\O)$ there exists a sequence 
$(v_\veps)_{\veps>0}\subset W^{1,1}(\O)\cap C(\overline{\O})$ 
such that $v_\veps \to v$ in $L^1(\O)$
and $\|\nabla v_\veps\|_{L^1(\O)} \to |Dv|(\O)$, cf.~\cite{Ambrosio2000,Attouch2006}.
With this we define an extension of the nodal interpolation operator 
$\cI_h: C(\overline{\O}) \to \cS^1(\cT_h)$ via
\[
\cI_h v = \lim_{\veps \to 0} \cI_h v_\veps,
\]
possibly after selection of a subsequence. If $d=1$ then we have the
nodal interpolation estimates 
\begin{equation}\label{eq:interpol_est}
\|v - \cI_h v\|_{L^1(T)} \le c h_T^r \| v^{(r)}\|_{L^1(T)}
\end{equation}
for $T\in \cT_h$, $v\in W^{r,p}(T)$, and $1\le r\le 2$,
cf.~\cite{BreSco08-book}. Via a limit passage for $\veps \to 0$ it follows
that the estimate holds for $r=1$ and $v\in BV(T)$ and the right-hand side
$c h_T |Dv|(T)$. A straightforward calculation reveals the total-variation diminishing
property of the nodal interpolation operator and its extension to 
$BV(\O)$, i.e., 
\[
\|(\cI_h v) '\|_{L^1(\O)} \le |Dv|(\O)
\]
for $v\in BV(\O)$ if $\O\subset \R$. This estimates fails 
in higher dimensional situations, cf.~\cite{BaNoSa15}. As a consequence of
Jensen's inequality, the Crouzeix--Raviart quasi-interpolation operator
$\cJ_h^{cr}: W^{1,1}(\O) \to \cS^{1,cr}(\cT_h)$ satisfies the discrete variant
\[
\|\nabla_h \cJ_h^{cr} v\|_{L^1(\O)} \le \|\nabla v\|_{L^1(\O)}
\]
for every $v\in W^{1,1}(\O)$. Note that the left-hand side of the inequality
does not coincide with the total variation of $\cJ_h^{cr}v$ as jump terms are
excluded. From a Poincar\'e inequality we deduce that 
\begin{equation}\label{eq:cr_est_l1}
\|v-\cJ_h^{cr}v \|_{L^1(T)}  \le c h_T \|\nabla v\|_{L^1(T)}
\end{equation}
for all $T\in \cT_h$. The operator and estimates can be extended to
$v\in BV(\O)$.


\section{Irregular solution}\label{sec:sing_sol}
The construction of a function $g\in L^2(\O)$ that leads to dual solutions which
are not Lipschitz continuous uses an idea from \cite{tovey2020} and the function 
\[
g = \chi_{B_r^+} - \chi_{B_r^-}
\]
in a domain $\O$ that is asymmetric with respect to the $x_2$ axis. 
Without making use of the co-area formula we
show here that the example can be understood by using uniqueness of minimizers
for $I$ and resulting antisymmetry properties. Hence, the problem reduces to 
a simpler problem for which the solution can be explicitly derived. 
We impose Dirichlet conditions on $u$ which eliminates
explicit boundary conditions from the dual problem. 

\begin{proposition} \label{2 disc solution}
Assume that $\O\subset \R^2$ is symmetric with respect to the $x_2$ axis, i.e., 
$\O = \O^+ \cup L_0 \cup \O^-$, where $\O^-$ is the reflection of $\O^+$
along the $x_2$ axis and $L_0$ is part of this axis, and assume that $r>0$ is such that
$\overline{B_r^\pm} \subset \O$. Then for $g=\chi_{B_r^+} - \chi_{B_r^-}$ the minimizer 
$u\in BV(\O)\cap L^2(\O)$ for $I$ subject to $u|_{\p\O}= 0$ is given by 
\[
u = c_{r,\a} g,  \quad c_{r,\a} = \max\{1-2/(\alpha r),0\}.
\]
\end{proposition}

\begin{proof}
Given a minimizer $u\in BV(\O)\cap L^2(\O)$ 
we define its antisymmetric reflection by $\hu(x_1,x_2) = -u(-x_1,x_2)$ and note that by the antisymmetry
of $g$ we have $I(u) = I(\hu)$. By convexity of $I$ we have for $\tu = (u+\hu)/2$
that $I(\tu) \le I(u)$ and conclude by uniqueness that $u = \tu = \hu$, i.e.,
$u$ is antisymmetric in $x_1$. Since
the jump set $J_u$ of $u$ is contained in the jump set $J_g$ of $g$, 
cf. \cite{Jump-set-paper}, we find that $u$ is continuous with value~$0$ in 
$L_0\setminus \{0\}$. It thus suffices to consider the reduced minimization of $I^+$ 
on $\O^+$ with $g^+= \chi_{B_r^+}$ subject to $u^+|_{\p\O^+}=0$. The solution is given
by $u^+ = c_{r,\a} g^+$, cf., e.g., \cite{CCCNP10,Bartels2015}, which 
implies the assertion. 
\end{proof}

The explicit representation of the solution $u$ implies properties of
dual solutions. The idea for proving the failure of Lipschitz
continuity is illustrated in Figure~\ref{fig:sketch_dual}.
We incorporate the characterization of dual solutions from~\cite{Jump-set-paper}.

\begin{corollary} \label{2 disc dual relations}
Assume that $\a r>2$. 
Any dual solution $z\in W^2(\diver;\O)$ for the setting considered in  
Proposition~\ref{2 disc solution} satisfies $\diver z= \alpha (u-g)$, i.e., 
$\diver z = -(2/r) g$, and $z \in \p|Du|$, i.e., on $\p B_r^+ \cup \p B_r^-$ we have
\[
z = \mp \nu^\pm
\]
with the outer unit normals $\nu^{\pm}$ on $\p B_r^\pm$. In particular, 
$z$ is not $\theta$-H\"older continuous at $x=0$ for every $\theta >1/2$. 
\end{corollary}

\begin{figure}[ht]
\scalebox{.9}{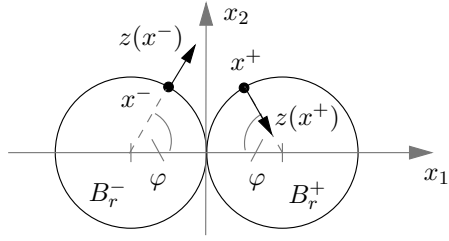}
\caption{\label{fig:sketch_dual}
Failure of Lipschitz continuity of a dual solution~$z$ at $x=0$ resulting
from $|x^+-x^-| \approx r \varphi^2$ and $|z(x^+)-z(x^-)| \approx 2 \varphi$
for $\varphi \to 0$.}
\end{figure}

\begin{proof}
Given an angle $\varphi \in (0,\pi/2)$ we define points $x^\pm \in \p B_r^\pm$
by 
\[
x^+ = r(- \cos \varphi +1, \sin \varphi), \quad x^- = r(\cos \varphi - 1, \sin \varphi).
\]
Any dual solution $z$ satisfies 
\[
z(x^+) = (\cos \varphi, -\sin \varphi), \quad 
z(x^-) = (\cos \varphi, \sin \varphi).
\]
We show that the modulus of the $\theta$-difference quotient $L_\varphi$ for $z$ 
defined with $x^+$ and $x^-$ is unbounded as $\varphi \to 0$, i.e., we have 
\[
L_\varphi = \frac{|z(x^+) - z(x^-)|}{|x^+ - x^-|^\theta } 
=  \frac{\sin \varphi}{r^\theta (1- \cos \varphi )^\theta}.
\]
Using the Taylor expansions $\sin \varphi = \varphi + O(\varphi^3)$ and
$1-\cos \varphi = \varphi^2/2 + O(\varphi^4)$ or l'Hospital's rule 
we find that $L_\varphi \to \infty$ as $\varphi \to 0$ whenever $\theta>1/2$. 
\end{proof}


\section{Graded meshes}\label{sec:graded}
To improve convergence results for discontinuous solutions, we use a
graded grid approach. This capitalizes the precise
observations about the approximation of a discontinuity by linear finite
elements. 

\begin{proposition} \label{4: Problem_theorem}
For $\ell> 0$ let $\O =(-\ell,\ell) \subset \R$ and $u(x) = \sign(x) + v(x)$ 
with a function $v\in W^{1,2}(\O)$. Let $u_h^* = \Pi_h u \in \cS^1(\cT_h)$ be the 
$L^2$ projection of $u$ onto $\cS^1(\cT_h)$. Assume that the triangulation 
$\cT_h$ is symmetric with respect to $x=0$. Then we have 
\[
c_1 h_0^{1/2} \le \|u-u_h^* \| \le c_2 \big(h_0^{1/2} + h_{\max} \|v'\|\big),
\]
where $h_0$ is the length of the elements containing the origin. 
\end{proposition}

\begin{proof}
We  decompose $v=a+s$ into antisymmetric and symmetric parts 
$a,s\in W^{1,2}(\O)$ with $a(0)=0$. We then have by $L^2$ orthogonality that
\[
\|u-u_h^*\|^2 = \min_{a_h,s_h\in \cS^1(\cT_h)} \|\sign + a - a_h\|^2 + \|s-s_h\|^2,
\]
where $a_h,s_h$ are antisymmetric and symmetric, respectively. In
particular, it suffices to consider the interval $(0,\ell)$ and the
restriction $a_h(0)=0$. The upper bound for the approximation error 
is obtained by defining $a_h$ via the nodal values of $1+a$ except for 
the origin and choosing $s_h$ as the nodal interpolant of $s$. We then find
that 
\[\begin{split}
\|u-u_h^*\| &\le 2^{1/2} h_0^{1/2} + \|a-\cI_h a\| + \|s-\cI_h s\| \\
&\le 2^{1/2} h_0^{1/2} + c_{\cI} h_{\max} \big(\|a'\| + \|s'\|\big).
\end{split}\]
To verify the lower bound we consider the contribution from the 
element $T_0=[0,h_0]$ to the antisymmetric part. Since $a_h(0)=0$, we 
have $a_h(x) = c x$ on $T_0$ and find that
\[
\|\sign + a - a_h\|^2  \ge \|1 + a - c x\|_{L^2(0,h_0)}^2.
\]
For the optimal slope $c$ we have Pythagoras' identity 
\[
\|1+a-cx\|_{L^2(0,h_0)}^2  = \int_0^{h_0} (1+a)^2 \dv{x} - \|cx\|_{L^2(0,h_0)}^{2}. 
\]
Since $1+a$ and $x$ are linearly independent
there exists $0\le \gamma_a<1$ with 
\[
c \|x\|_{L^2(0,h_0)}^2 = \int_0^{h_0} (1+a) x \dv{x}
\le \gamma_a \|1+a \|_{L^2(0,h_0)} \|x\|_{L^2(0,h_0)}.
\]
The continuity of $a$ and $a(0)=0$ show that
$\|1+a\|_{L^2(0,h_0)} \ge c_a h_0^{1/2}$ which implies that we have
$\|1+a-cx\|_{L^2(0,h_0)}^2 \ge (1-\g_a^2) c_a^2 h_0$.
\end{proof}

Meshes with a grading towards a given point are obtained from affine
mappings of a graded grid of the reference interval to a macro element. 

\begin{definition}\label{def graded grid}
A {\em $\b$-graded grid} of the reference element $\hT = [0,1]$ is
for $J\in \N$ and $\b \ge 1$ given by the points
\[
0 = \xi_0 < \xi_1 < \dots < \xi_J = 1
\]
with $\xi_j = (j/J)^\b$. The length of the interval $[\xi_{j-1},\xi_j]$
is bounded by $J^{-\b} \b j^{\b-1}$, in particular, we have 
$h_{\min} = J^{-\b}$, $h_{\max} = \b J^{-1}$, and $h = J^{-1}$. 
\end{definition}

Proposition~\ref{4: Problem_theorem} implies that, for
a grading strength $\b= 2$ towards the origin, so that $h_0 = h^2$ and
$h_{\max} \le c h$, we obtain the linear rate 
\[
\|u- u_h^*\| \le c h.
\]
For the approximation of the ROF problem a similar result can be obtained. 

\begin{proposition}\label{prop:error_graded_1d}
Let $\O \subset \R$, $g\in L^\infty(\O)$, and assume that the 
minimizer $u\in BV(\O)\cap L^2(\O)$ of $I$ is a piecewise $W^{2,1}$ function. 
Then if $\cT_h$ is graded with strength $\b$ towards the jumps $J_u$ of~$u$
we have for the $P1$ finite element minimizer $u_h\in\cS^1_D(\cT_h)$ for $I$
that
\[
\|u-u_h\|^2 \le c \big(h^\b|Du|(\O) 
+ \b^2 h^2 \|u''\|_{L^1(\O\setminus J_u)} \big) \|g\|_{L^\infty(\O)}.
\]
\end{proposition}

\begin{proof}
By minimality of $u$ and coercivity properties of $I$, we find that
for every $\tu_h\in \cS^1_D(\cT_h)$ we have
\[
\frac{\a}{2} \|u-u_h\|^2 \le I(u_h) - I(u) \le I( \tu_h) - I(u).
\]
Via regularization we set $\tu_h  = \cI_h u$  and 
use the total-variation diminishing property of~$\cI_h$ 
to deduce with a binomial formula that
\[\begin{split}
\frac{\a}{2} \|u-u_h\|^2 &\le \int_\O (\tu_h - g)^2 - (u-g)^2 \dv{x} \\
&\le \|\tu_h-u\|_{L^1(\O)} \|\tu_h + u -2 g\|_{L^\infty(\O)}.
\end{split}\]
We note that $\|\tu_h\|_{L^\infty(\O)} \le \|u\|_{L^\infty(\O)} \le \|g\|_{L^\infty(\O)}$
and decompose the first factor into elementwise contributions. 
If $T\cap J_u \neq \emptyset$ we have 
\[
\|\tu_h - u\|_{L^1(T)} \le c h_{\min} |Du|(T).
\]
Otherwise, standard interpolation estimates show that 
\[
\|\tu_h - u\|_{L^1(T)} \le c h_{\max}^2 \|u''\|_{L^1(T)}.
\] 
A summation over the elements leads to the asserted estimate. 
\end{proof}

\begin{remark}
If the solution $u$ is piecewise linear then the estimate can be
improved to the convergence rate $O(h^{\b/2})$ for every $\b \ge 1$. 
\end{remark}

Because the total-variation diminishing property is not satisfied 
a generalization of the argument to a higher-dimensional setting requires additional
assumptions on a dual solution. A key difficulty is that the construction of a
discrete dual variable in the space $\RT^0_N(\cT_h)$ leads to a local violation of
the constraint $|z_h(x_T)|\le 1$ of order $O(h_T)$ independently of additional regularity
properties. For piecewise constant primal solutions we assume the strict inequality
$|z(x)|<1$ with linear decay away from discontinuities so that no violation occurs,
while in a neighborhood of the jump set we have a quadratic violation $O(h^2)$
on suitably graded meshes. 

\begin{proposition}
Let $g\in L^\infty(\O)$ and
assume that the primal solution $u\in BV(\O)\cap L^\infty(\O)$ is piecewise
constant with piecewise regular jump set $J_u$, and there exists a
dual solution $z\in W^1_N(\diver;\O)\cap W^{1,\infty}(\O;\R^d)$ with Lipschitz
constant $L\ge 0$ and such that there exists $\ell>0$ with
\[
|z(x)| \le 1 - \ell d(x)
\]
where $d(x)  = \inf_{y\in J_u} |x-y|$. Let $(\cT_h)_{h>0}$ be a sequence of 
quadratically graded triangulations towards $J_u$, i.e., with
$d(T) = \inf_{x\in T} d(x)$ we have for all $T\in \cT_h$ that 
\[
h_T \le c \begin{cases}
h \, d(T)^{1/2} & \text{if } d(T) \ge h^2, \\
h^2 & \text{otherwise}. 
\end{cases}
\] 
Then, the 
Crouzeix--Raviart finite element solution $u_h\in \cS^{1,cr}_D(\cT_h)$ 
of the ROF model satisfies
\[
\|\Pi_h(u-u_h)\| \le h |\log h|  \, M(\a,u,z,g,\s).
\]
\end{proposition}

\begin{proof}
We follow the strategies of~\cite{Chamb,Bart20a}, let $g_h =\Pi_h g$, and note that 
there exists a Crouzeix--Raviart quasi-interpolant $\tu_h \in \cS_D^{1,cr}(\cT_h)$ of~$u$ 
such that
\[
I_h(\tu_h) \le I(u) + \frac{\a}{2}\|\Pi_h \tu_h - u\|_{L^1(\O)} 4 \|g\|_{L^\infty(\O)} 
- \frac{\a}{2} \|g-g_h\|^2.
\]
If $\g_h \ge 1$ is such that for the corrected 
Raviart-Thomas quasi-interpolant $\tz_h = \g_h^{-1} \cJ_{RT}z \in \RT^0_N(\cT_h)$
of $z$ we have $|\tz_h(x_T)|\le 1$ for all $T\in \cT_h$ then we have
\[
D_h(\tz_h) \ge D(z) - (1-\g_h^{-1}) \|g\| \|\diver z\| - \frac{\a}{2} \|g-g_h\|^2.
\]
Using the coercivity of $I_h$, the minimality of $u_h$, the discrete
duality relation $I_h(u_h) \ge D_h(\tz_h)$, and the strong duality relation
$I(u) = D(z)$, we find that
\[\begin{split}
\frac{\a}{2} \|\Pi_h(\tu_h-u_h)\|^2 & \le I_h(\tu_h)-I_h(u_h)  \le I_h(\tu_h) - D_h(\tz_h) \\
& \le \frac{\a}{2}\|\Pi_h \tu_h - u\|_{L^1(\O)} 4 \|g\|_{L^\infty(\O)}
+ (1-\g_h^{-1}) \|g\| \|\diver z\|.
\end{split}\]
To bound the correction factor $\g_h$ we estimate $\g_T = |\cJ_{\RT} z (x_T)|$
for every $T\in \cT_h$ using that $\g_T \le |z(x_T)| + c_{RT} h_T L$.
A case distinction shows that given $c_0>0$ there 
exist $c,h_0>0$ such that for $0<h<h_0$ we have 
\[
h_T \le \begin{cases}
c_0 d(T) & \mbox{if } d(T)\ge h^2 |\log h|^4, \\
c h^2 |\log h|^2  & \mbox{otherwise}.
\end{cases} 
\]
Let $c_0 \le \ell/(L c_{RT})$. If $d(T) > h^2 |\log h|^4$, then we have
\[
\g_T  \le 1- \ell d(T) + c_{RT} h_T L \le 1,
\]
and otherwise 
\[
\g_T \le 1 + c h^2 |\log h|^2 L.
\]
It thus suffices
to choose $\g_h \le 1+c h^2 |\log h|^2$ so that $(1-\g_h^{-1}) \le c h^2 |\log h|^2$.
To bound the term $\|\Pi_h \tu_h - u\|_{L^1(\O)}$,
we note that if $T\cap J_u \neq \emptyset$ we obtain with~\eqref{eq:cr_est_l1} that 
\[
\|u-\Pi_h \tu_h\|_{L^1(T)} \le \|u- \tu_h\|_{L^1(T)} + \|\tu_h-\Pi_h \tu_h\|_{L^1(T)}
\le  c h_T  |Du|(T).
\]
Otherwise, if $T\cap J_u = \emptyset$ we have that $u$ is constant
and $u=\Pi_h \tu_h$ on $T$. The estimate of the proposition follows from a combination
of the estimates and the triangle inequality, noting that by Jensen's and
H\"older's inequalities we have 
$\|\Pi_h(u-\tu_h)\|^2 \le \|u-\tu_h\|_{L^1(\O)} \|u-\tu_h\|_{L^\infty(\O)}$. 
\end{proof}

\begin{remark}
Under additional conditions on the approximations $(u_h)_{h>0}$,
e.g., that they are uniformly bounded and piecewise $W^{1,2}$, 
the estimate of the proposition also applies to the full $L^2$ error $\|u-u_h\|$.
\end{remark}

The assumptions of the proposition apply to certain 
settings with piecewise constant solutions. 

\begin{example}
If $g= \chi_{B_r(0)}$ for $r>0$ with $\overline{B_r(0)} \subset \O$ and 
if Dirichlet conditions on $u$ are imposed on $\GD = \p\O$, then we 
have $u=c_{r,\a} g$ and 
\[
z(x) = -c_{r,\a}' \begin{cases}
r^{-1} x, & |x|\le r, \\
r x/|x|^2, & |x| \ge r,
\end{cases}
\]
for every $x\in \O$, where $c_{r,\a}' = \min\{1,r\a/d\}$, cf.~\cite{CCCNP10,Bartels2015}.
\end{example}

A quadratic grading is the optimal
grading strength to locally refine a two-dimensional triangulation
towards a one-dimensional subset. 

\begin{example} \label{simple grid}
Let $\O = (0,\ell)^2 \subset \R^2$ with $\ell>0$ with initial
triangulation $\cT_0 = \{T_1,T_2\}$. 
We inductively define $\cT_{k+1}$ by first applying a red refinement
to all elements in $\cT_k$ that intersect the $x_2$-axis and then 
refining further elements to avoid hanging nodes by a red-green-blue
refinement strategy as, e.g., in~\cite[Algorithm 4.2]{Bartels2016}, 
cf.~Figure~\ref{fig: selfadapted triangs}.  
\end{example} 

\begin{figure}[h!]
\scalebox{.8}{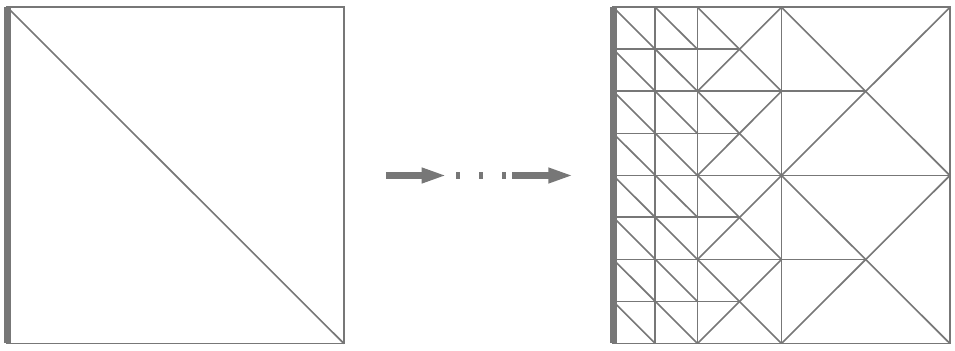}
\caption{\label{fig: selfadapted triangs} 
Triangulation $\cT_3$  in Example~\ref{simple grid} 
obtained by refinements of $\cT_0$ 
with a grading towards the side $\{0\}\times [0,\ell]$.}
\end{figure} 

We define the (asymptotic) {\em grading strength} of a sequence of regular triangulations
$(\cT_k)_{k \ge 0}$ as the logarithmic relation of the minimal and average mesh-size,
i.e., 
\[
\b = \lim_{k\to \infty} \frac{\log(h_{k,\min})}{\log(h_k)}.
\]
We note that the speed of convergence depends on the diameter of $\O$, e.g,. 
for the triangulations defined in Example~\ref{simple grid} we have 
that $h_{k,\min}$ and $h_k$ are proportional to  $\ell$
which is irrelevant in the limit passage. 
For the generic setting of Example~\ref{simple grid} we identify a
quadratic grading strength. 

\begin{proposition}\label{beta convergence}
Let $(\cT_k)_{k=0,1,\dots}$ be a sequence of triangulations of $(0,\ell)^2$
such that triangles along the side $\{0\}\times [0,\ell]$ are $k$-times
refined with $h_T \sim h_{k,\min} \sim q^k h_{k,\max}$ and for 
triangles along the side $\{\ell\}\times [0,\ell]$ we have $h_T = h_{k,\max}$.
If $k^{-1} \log(h_{k,\max}) \to 0$ then the graduation strenth is $\b=2$. 
\end{proposition}

\begin{proof}
To determine the average
mesh-size we note that the refinement process defines after $k$-steps
a partitioning into stripes $S_j$, $j=0,1,\dots,k$, with $n_j \sim q^{-j} h_{\max}$ 
elements. A summation shows that $\cT_k$ contains $N_k \sim q^{-k} h_{\max}$ 
elements so that $h_k \sim q^{k/2}$ and
\[
\b = \lim_{k\to \infty} \frac{\log(c q^k h_{k,\max})}{\log(c' q^{k/2} h_{k,\max})}
= \lim_{k\to \infty} \frac{\log(c h_{k,\max}) + k \log(q)}{\log(c' h_{k,\max}) + (k/2) \log(q)} = 2,
\]
which proves the assertion. 
\end{proof}

\section{Primal-dual gap estimator via discrete duality}\label{sec:prim-dual}
We next devise a strategy that leads to an automatic and adaptive 
local mesh refinement algorithm. To illustrate the main ideas we consider a 
general convex minimization problem
\[
\tI(u) = \int_\O \phi(\nabla u) + \psi(x,u) \dv{x}
\]
defined on a Sobolev space $X=W^{1,p}_D(\O)$, $1<p<\infty$, or on $X=BV(\O)$,
whose dual is given by the maximization of 
\[
\tD(z) = -\int_\O \phi^*(z) + \psi^*(x,\diver z) \dv{x}
\]
on a space of vector fields $W=W^{p'}_N(\diver;\O)$. Here, $\phi: \R^d \to \Rinf$
and $\psi:\O \times \R \to \Rinf$ are convex functionals and  
$\phi^*$ and $\psi^*$ are their convex conjugates. The duality
relation $\tI(u) \ge \tD(z)$ in combination with 
coercivity properties of $\tI$ described by a functional $\s_{\tI}$
imply, for the minimizer $u\in X$ and arbitrary $v \in X$ and $q \in W$ that 
\[
\s_{\tI}^2(u,v) \le \tI(v)-\tI(u) \le \tI(v) - \tD(q) = \eta_h^2(v,q).
\]
If $v=u_h$ for an approximation
$u_h \in X$ of~$u$, then $\eta_h(u_h,q)$  provides a computable bound on 
the approximation error $\s(u,u_h)$ whenever an admissible $q$ is
explicitly given. We use the following extended result from~\cite{Bart15}.

\begin{proposition}\label{numerical_estimator}
Let $u \in BV(\O) \cap L^2(\O)$ be the minimizer for the ROF functional
$I$ and $u_h\in \cS^1_\cD(\cT_h)$ an approximation. We then have, for every
$q \in W^2_N(\diver;\O)$ with $|q| \leq 1$ in $\O$, that 
\[
\|u - u_h \| \le \big(2/\a\big)^{1/2} \,  \eta_h(u_h,q) + \|g -\Pi_h g\|,
\]
where 
\[
\eta_h^2(u_h,q) = \int_\O |\nabla u_h| - \nabla u_h \cdot \Pi_h q \dv{x}
+ \frac{1}{2\a} \int_\O \big(\diver q - \a (u_h -g_h) \big)^2 \dv{x}.
\]
\end{proposition} 

\begin{proof}
We define $g_h = \Pi_hg$ and let $\tI$ be the ROF functional with
$g$ replaced by $g_h$ whose minimizer we denote by $\tu\in BV(\O)\cap L^2(\O)$. 
By the strong convexity of the $L^2$ term in $I$ we find that
$\|u-\tu\| \le \|g-g_h\|$, cf., e.g.,~\cite{Bartels2015}. 
We apply the error estimate to $\tI$ and obtain that
\[
\frac{\a}{2} \|\tu- u_h \|^2 \le 
 \tI(u_h) - \tD(q) - \int_\O \nabla u_h \cdot q \dv{x} - \int_\O u_h \diver q \dv{x}.
\]
A straightforward calculation, the fact that $\nabla u_h$ is elementwise
constant, and the triangle inequality lead to the formula for $\eta_h(u_h,q)$.
\end{proof}

\begin{remark}
If the estimate is derived for a Crouzeix--Raviart approximation 
$u_h\in \cS^{1,cr}_D(\cT_h)$, then jumps across sides occur on the
right-hand side.
\end{remark}

The optimal estimator $\eta_h(u_h,q)$ requires an exact solution of
the dual problem or a numerical approximation of sufficient accuracy, 
cf.~\cite{BarMil20}.
Since the numerical solution of the dual problem is computationally 
expensive, we aim at the construction of a nearly optimal approximation
at a computational cost that is comparable to the numerical solution 
of the discretized primal problem. For this we use a reconstruction
of a discrete dual solution from the Crouzeix--Raviart approximation
of the primal problem from~\cite{Bart20a}.

\begin{proposition}[{\cite[Proposition 3.1]{Bart20a}}]
Let $\tI_h$ and $\tD_h$ be defined on $\cS^{1,\CR}_D(\cT_h)$ and $\RT^0_N(\cT_h)$ 
with $\psi_h(\cdot,a) = \Pi_h \psi (\cdot,a)$ for all $a\in \R$ via
\[\begin{split}
\tI_h(u_h) &= \int_\O \phi(\nabla_h u_h) + \psi_h(x,\Pi_hu_h) \dv{x}, \\
\tD_h(z_h) &= -\int_\O \phi^*(\Pi_h z_h) + \psi_h^*(x,\diver z_h) \dv{x}.
\end{split}\]
We then have the duality relation $\tI_h(u_h) \ge \tD_h(z_h)$.
If $s\mapsto \phi(s)$ and $a\mapsto \psi_h(x,a)$ are continuously 
differentiable and if $u_h$ is minimal for $\tI_h$ then a maximizing element $z_h$
for $\tD_h$ is given by 
\[
z_h = \phi'(\nabla_h u_h) + d^{-1} \, \psi_h'(\cdot ,\Pi_h u_h) (\cdot  - x_\cT),
\]
where $x_\cT = \Pi_h \id$, and strong duality $\tI_h(u_h) = \tD_h(z_h)$ applies. 
\end{proposition}

To apply the result to the discretized ROF functional, we consider for $\veps>0$
the regularization $|a|_\veps = \big(|a|^2+ \veps^2)^{1/2}$ 
of the non-differentiable modulus.
We then obtain the reconstruction $z_h \in \RT^0_N(\cT_h)$ given by
\[
z_h = \frac{\nabla u_h}{|\nabla_h u_h|_\veps} 
+ \frac{\a}{d} \, \Pi_h (u_h - g) (\cdot - x_\cT).
\]
The vector field $z_h$ satisfies $|z_h(x_T)| \le 1$ for all $T\in \cT_h$,
but in general not $|z_h(x)|\le 1$ for almost every $x\in \O$. We have
for every $T\in \cT_h$ that
\[
\big|z_h|_T\big| \le 1 + \frac{\a}{d} |(u_h-g_h)(x_T)| \frac{d}{d+1} h_T = \g_T.
\]
The globally re-scaled vector field $\hz_h = (\max_{T\in \cT_h} \g_T)^{-1} z_h$ 
satisfies $|\hz_h|\le 1$ but does not lead to an efficient 
error estimator. Our experiments reported below indicate that also the 
scaling $\tz_h = \tg_h^{-1} z_h$  with a continuous function $\tg_h$
satisfying $\tg_h|_T \ge \g_T$ for all $T\in \cT_h$ does not lead to an 
efficient estimator. 

\begin{remark}\label{rem:stronger_quant}
The error estimator $\eta_h(u_h,q)$ controls the approximation error
$\s_{\tI}(u,u_h)$ defined by the maximal coercivity of the functional~$\tI$. 
For the ROF functional $I$ the scaled $L^2$ norm is a lower bound
for this quantity and the error estimator controls a stronger error
quantity. For the regularized ROF functional $I^\veps$ 
and a minimizers $u$, i.e., $\d I^\veps(u) =0$, a Taylor expansion 
formally yields that
\[
I^\veps(u_h) 
= I^\veps(u) + 
\int_\O \phi_\veps''(\nabla \xi)[\nabla(u-u_h),\nabla(u-u_h)] \dv{x} + \frac{\a}{2} \|u-u_h\|^2,
\]
where the convex function $\phi_\veps = |\cdot|_\veps$ has a positive definite 
Hessian, cf.~\cite{Veser}. 
\end{remark}


\section{Numerical experiments}\label{sec:num_ex}
We verify in this section the theoretical results and investigate the
performance of numerical methods beyond their validity. Our computations
are based on the use of the regularized ROF functional 
\[
I^\veps(u) = \int_\O |\nabla u|_\veps \dv{x} + \frac{\a}{2} \|u-g\|^2
\]
with the regularized modulus $|a|_\veps = (|a|^2 + \veps^2)^{1/2}$ for $a \in \R^d$ 
and $\veps > 0$. Owing to the bounds $0 \leq |a|_\veps - |a| \leq \veps $ 
the error estimates and identified convergence rates remain valid 
provided that $\veps = O(h^\s)$ with $\s=1$ or $\s=2$ to obtain an $L^2$ 
error $O(h^{\s/2})$ on uniform and locally refined meshes. 
The iterative minimization of $I^\veps$ was
realized with the unconditionally stable semi-implicit $L^2$ gradient flow
from~\cite{BaDiNo18}. We always use the step-size $\tau=1$ but
different stopping criteria $\|u^k - u^{k-1}\|\le \veps_\stop$. 

\subsection{Irregular solution} \label{subsec: num_ex_Tovey}
We investigate the numerical approximation of the example
from Section~\ref{sec:sing_sol} to verify whether the 
failure of Lipschitz continuity of dual solutions affects the
convergence rate $O(h^{1/2})$ for the 
Crouzeix--Raviart method on uniform triangulations. We use a coordinate
transformation to avoid superconvergence phenomena related
to mesh symmetries.

\begin{example}[Non-Lipschitz dual]\label{example Tovey}
Let $\O = (-1,1)^2 \subset \R^2$, $\a=10$, and $\widetilde{g} = \chi_{B_r^+} - \chi_{B_r^-}$  
for $r\in \{0.4,5.0\}$ and $g= \widetilde{g} \circ \Phi$, where $\Phi(x) = Qx+b$
realizes a rotation by $\phi = 70^\circ$ and shift by $b=(0.1,0)^\transp$.
Dirichlet conditions $u_D = u|_{\p\O}$ from the solution 
$u=c_{r,\a} g$, $c_{r,\a} = 1-2/(r\a)$, are imposed.
\end{example}
 
The experimental convergence rates shown in Figure~\ref{fig:Tovey_conv} 
are obtained on $k$-times red-refined triangulations $\cT_k$ 
of an initial triangulation  $\cT_0$ with four elements. We have
$h_k = 2^{-k}$ and use $\veps_\stop = h_k/20$.
 The optimal convergence rate $O(h^{1/2})$ is observed
for both choices of $r=0.4$ and $r=5.0$, despite the lack of a Lipschitz
continuous dual solution. Also the scaling of the problem towards 
the singular point obtained by increasing
the radius $r$ does not affect the experimental convergence rates.
In Figure~\ref{fig:Tovey_pic} the numerical solution $u_h$ on the
triangulation $\cT_6$ and its projection onto elementwise constant
functions are displayed for the parameter $r=0.4$. Large gradients
occur near the origin, the midpoint values do not, however, show 
artifacts.  

\begin{figure}[p]
\includegraphics[width=9.6cm]{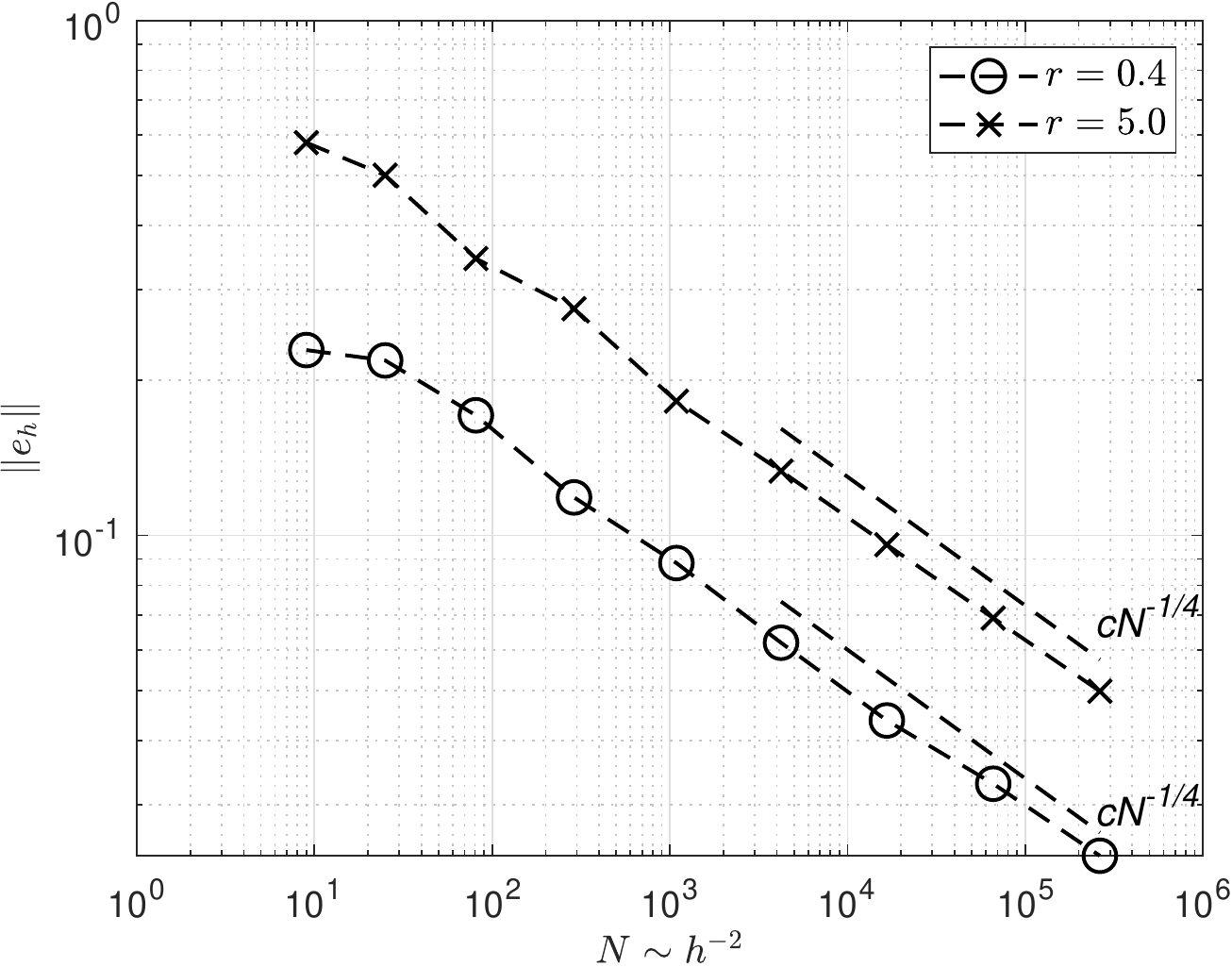} 
\caption{\label{fig:Tovey_conv} 
Experimental convergence rates $h^{1/2} \sim N^{-1/4}$ 
for Crouzeix--Raviart finite element approximations of the ROF model
on sequences of uniform triangulations for a solution 
with non-Lipschitz continuous dual solutions defined in 
Example~\ref{example Tovey} and different magnifications of 
the irregular region at the origin.}
\includegraphics[width = 6.2cm]{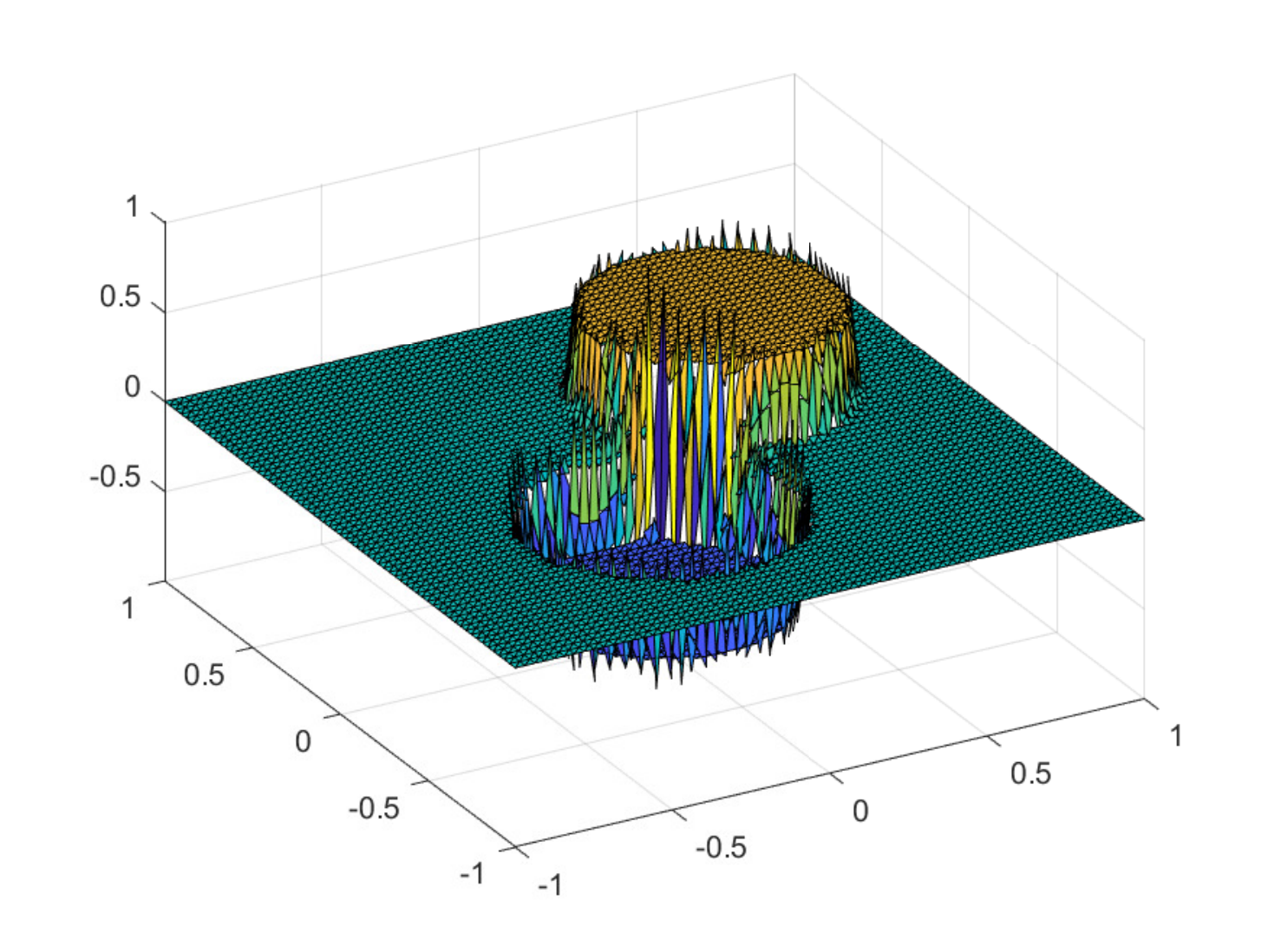}
\includegraphics[width = 6.2cm]{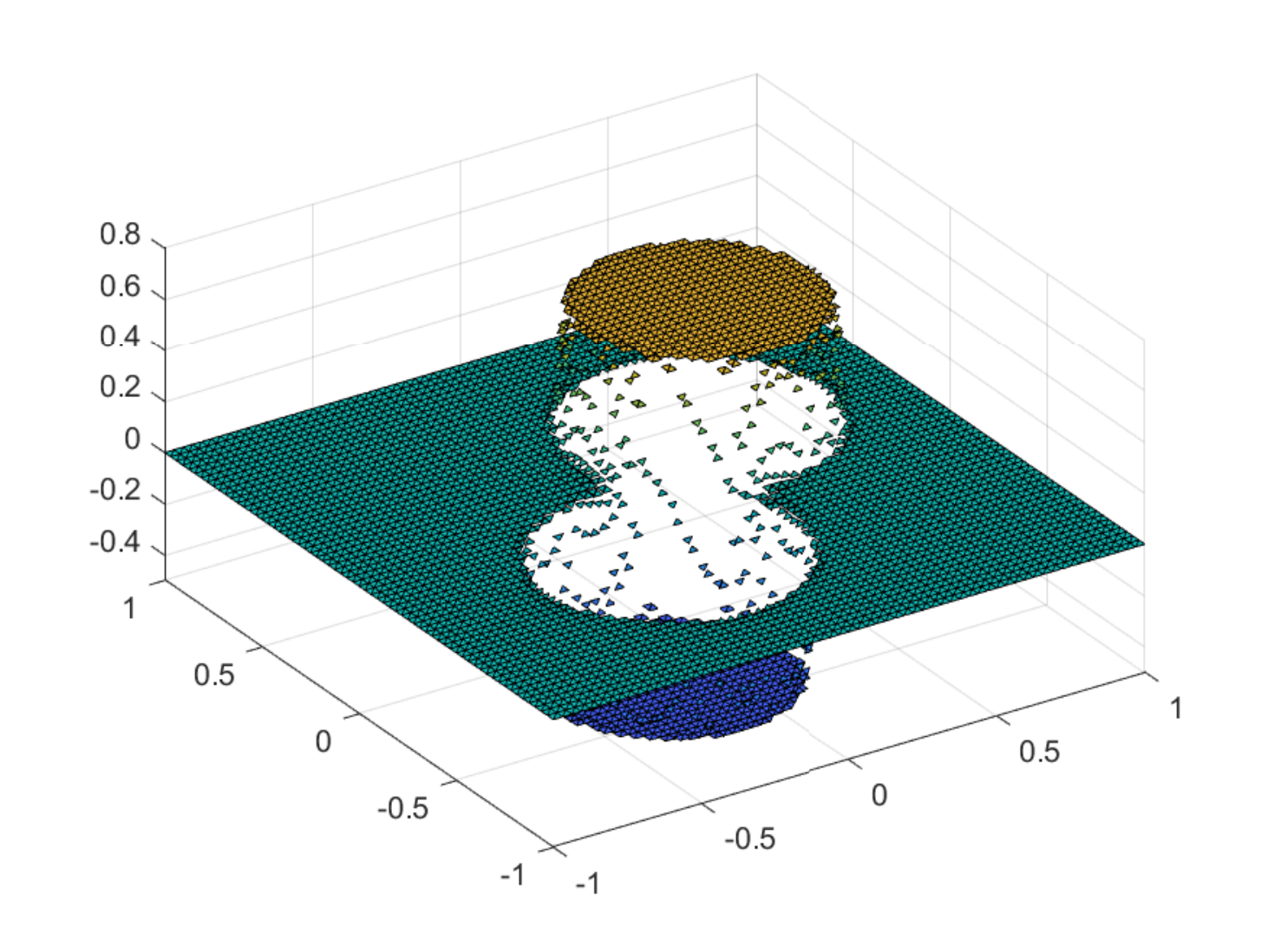}
\caption{\label{fig:Tovey_pic}
Numerical solution $u_h \in \cS^{1,cr}(\cT_6)$ (left) and its 
projection $\Pi_h u_h$ (right) in Example~\ref{example Tovey} for
$r=0.4$. Large discrete gradients occur near the origin where
dual solutions are not Lipschitz continuous.}
\end{figure}

\subsection{Mesh grading in one dimension}
We next confirm our theoretical findings for the use 
graded meshes for the approximation of the ROF model in 
one-dimensional settings. The problem specification leads
to a multiple of the sign function as exact solution. 

\begin{example}[1D sign function] \label{one disc 1d}
Let $\O = (-1,1)$, $\a = 10$, and define $g(x)= \sign(x)$.
The minimizer for the ROF functional subject to 
Dirichlet boundary conditions is given by  
$u = c_{r,\a} g$, $c_{r,\a} = (1-1/(r\a))$, for $r=2$. 
\end{example}

In our experiments we choose 
the regularization $\veps = h^{\beta}$ so that the corresponding
error contribution is of the same order as the discretization
error. We note that the stopping criterion has to be carefully
chosen and we used $\veps_\stop = h/20$ for $\b=1$ and 
the finer tolerance $\veps_\stop = h^{\beta +1}/20$ for non-uniform 
meshes with grading strength $\b>1$.
The experimental convergence rates obtained with these settings
for a $P1$ method are given in Figure~\ref{fig:graded_grid_1d_conv}.

\begin{figure}[p]
\includegraphics[width=9.6cm]{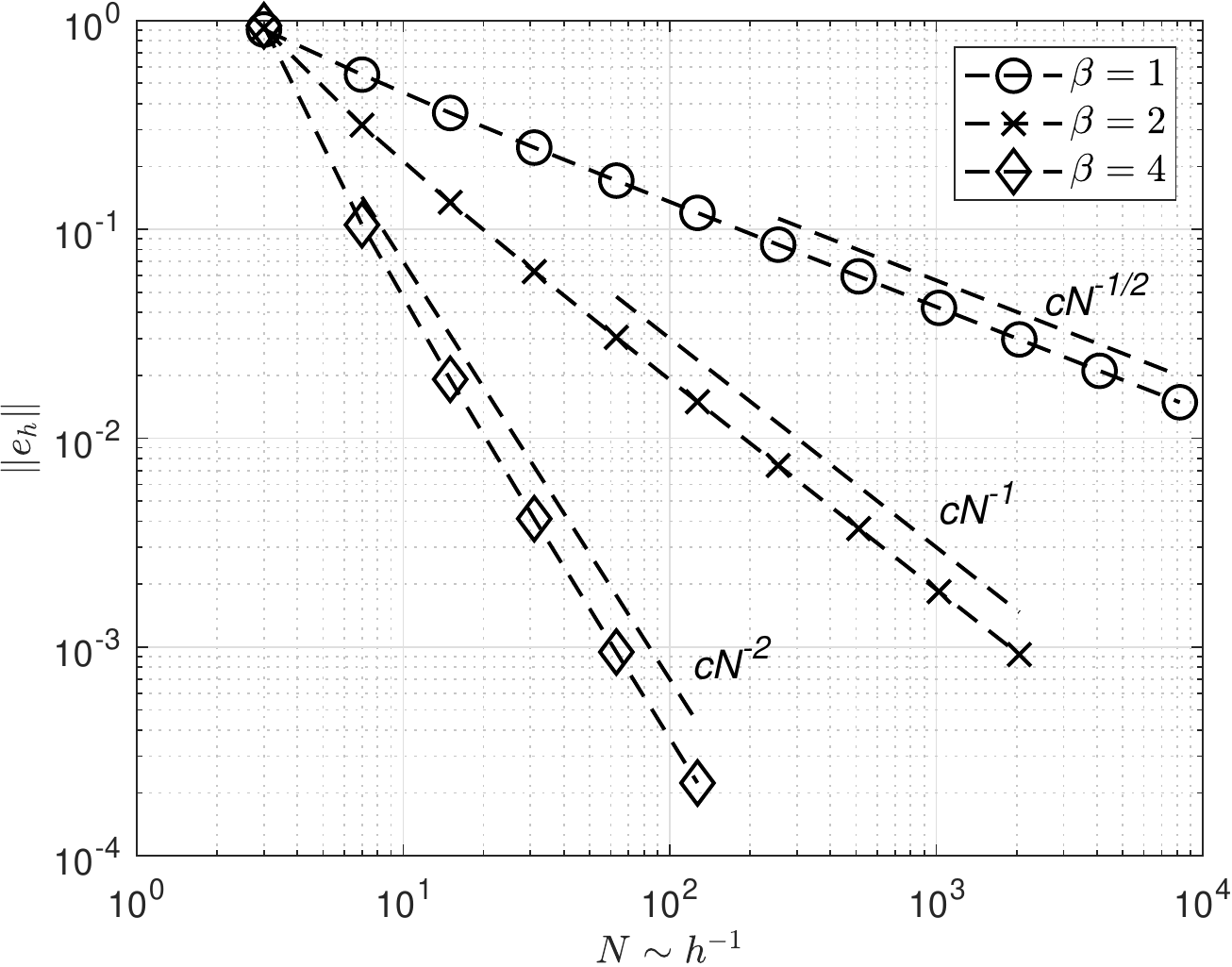} 
\caption{\label{fig:graded_grid_1d_conv} 
Experimental convergence rates in approximating the one-dimensional
ROF model defined in Example~\ref{one disc 1d} on meshes with a 
grading towards the discontinuity with different grading strengths 
$\b$ leading to convergence rates $h^{\b/2} \sim N^{-\b/2}$. }

\mbox{} \vspace{5mm} 

\includegraphics[width = 6.2cm]{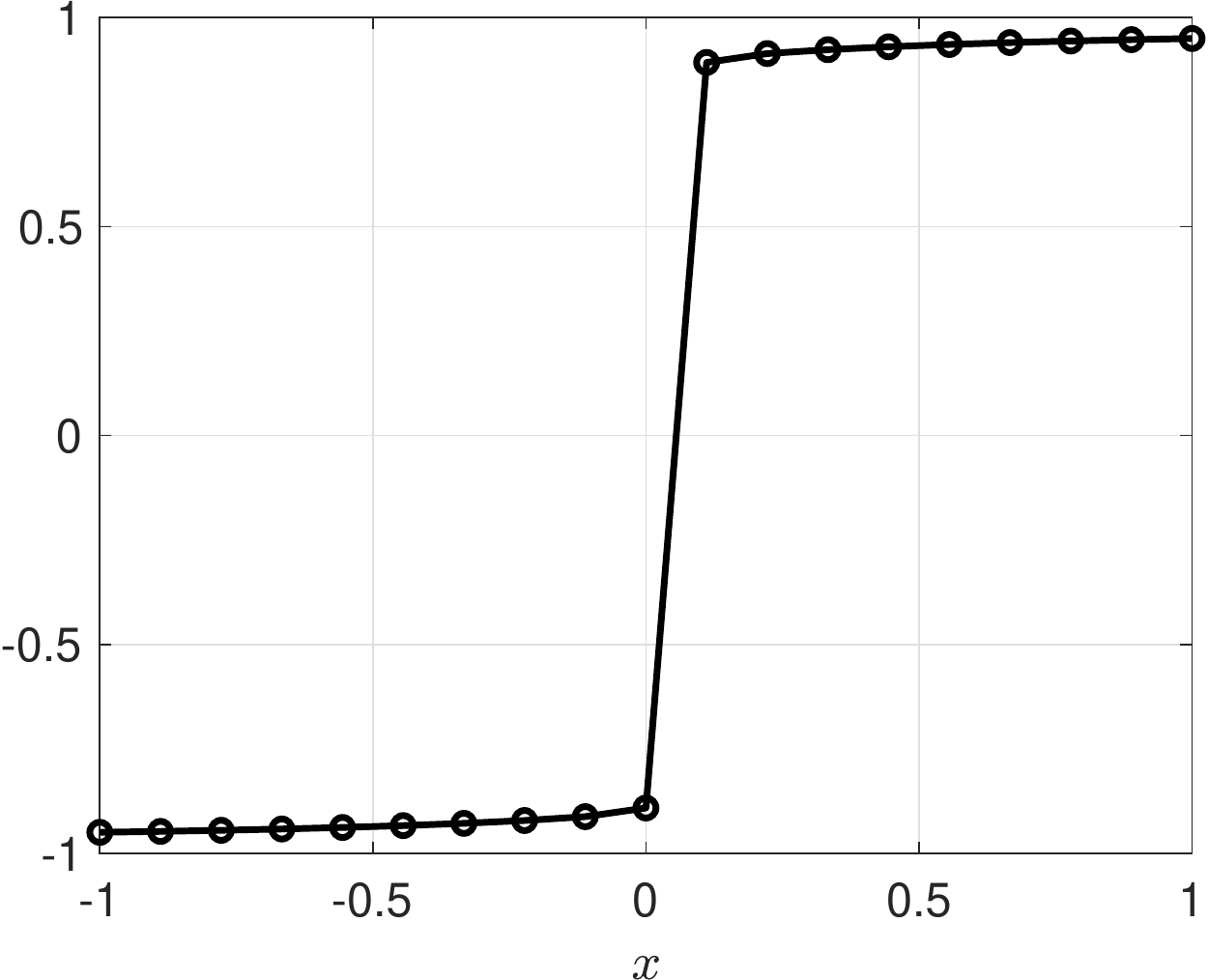}
\includegraphics[width = 6.2cm]{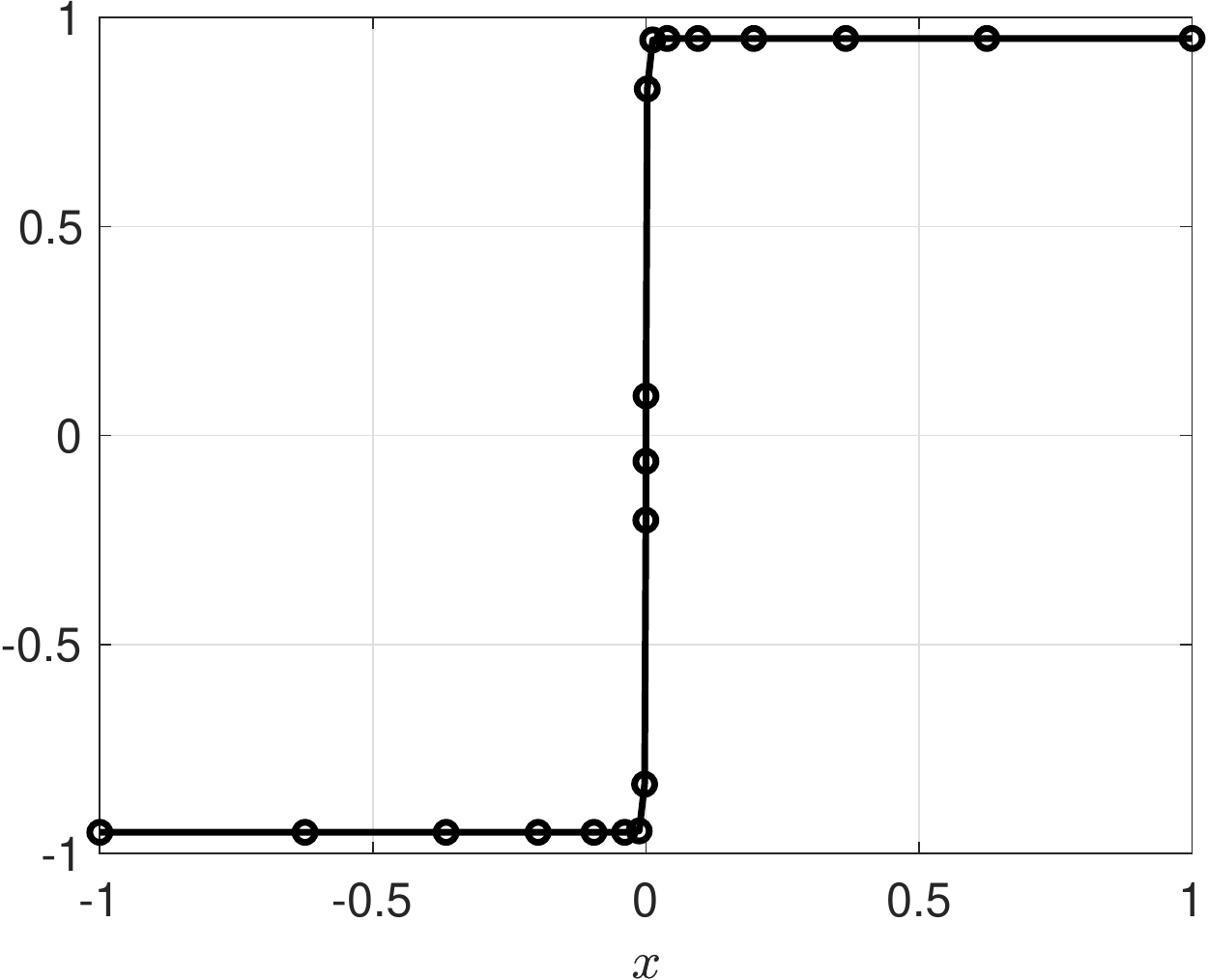} 
\caption{\label{fig:graded_grid_1d_pic} 
Numerical solutions on a uniform and a graded mesh in the
one-dimensional setting with piecewise constant solution
specified in Example~\ref{one disc 1d}. The strong grading
with $\beta = 4$ (right) leads to a high accuracy in comparison
with the uniform grid corresponding to $\b=1$ (left).}
\end{figure}

\subsection{Mesh grading in two dimensions}\label{subsec: numex graded_two}
We experimentally
investigate the performance of finite element
approximations for a standard example using mesh grading based
on the discontinuity set of the given function $g$. 

\begin{example}[Single disc phantom] \label{graded one disc}
Let $\O = (-1,1)^2$, $\a=10$, and  $g = \chi_{B_r(0)}$ for $r = 1/2$. 
For homogeneous Dirichlet boundary conditions the minimizer of
the ROF model is given by $u=c_{r,\a} g$ with $c_{r,\a} = 1-2/(r\a)$. 
\end{example}

Our initial triangulation $\cT_0$ consists of two right triangles 
that partition $\O$ and we iteratively define a sequence of regular 
triangulations $(\cT_k)_{k=0,1,\dots}$ by performing a red refinement
for all triangles in $\cT_k$ that have a non-empty intersection
with the discontinuity set $J_g = \p B_r(0)$ of $g$ and then
carrying out a red-green-blue refinement procedure to avoid hanging nodes. 
We verified that this leads to a quadratic grading strength. 
To allow for a nearly linear experimental convergence rate, we
choose $\veps_\stop = h^2/20$ and $\veps = h^2$. For the approximations
obtained with the Crouzeix--Raviart method we observe a nearly
linear experimental convergence rate. This is not the case for
approximations obtained with less flexible $P1$ finite elements, as can be seen
in Figure~\ref{fig: conv graded 2d}. We also illustrated the 
convergence behavior of the error estimator from Section~\ref{sec:prim-dual}
and observe that it serves as a reliable but non-efficient 
error bound. An explanation for this observation is that the
graded meshes are optimal for the $L^2$ error, but not
necessarily for the
error quantity controlled by the estimator, cf. Remark~\ref{rem:stronger_quant}.

\begin{figure}[p]
\includegraphics[width = 9.6cm]{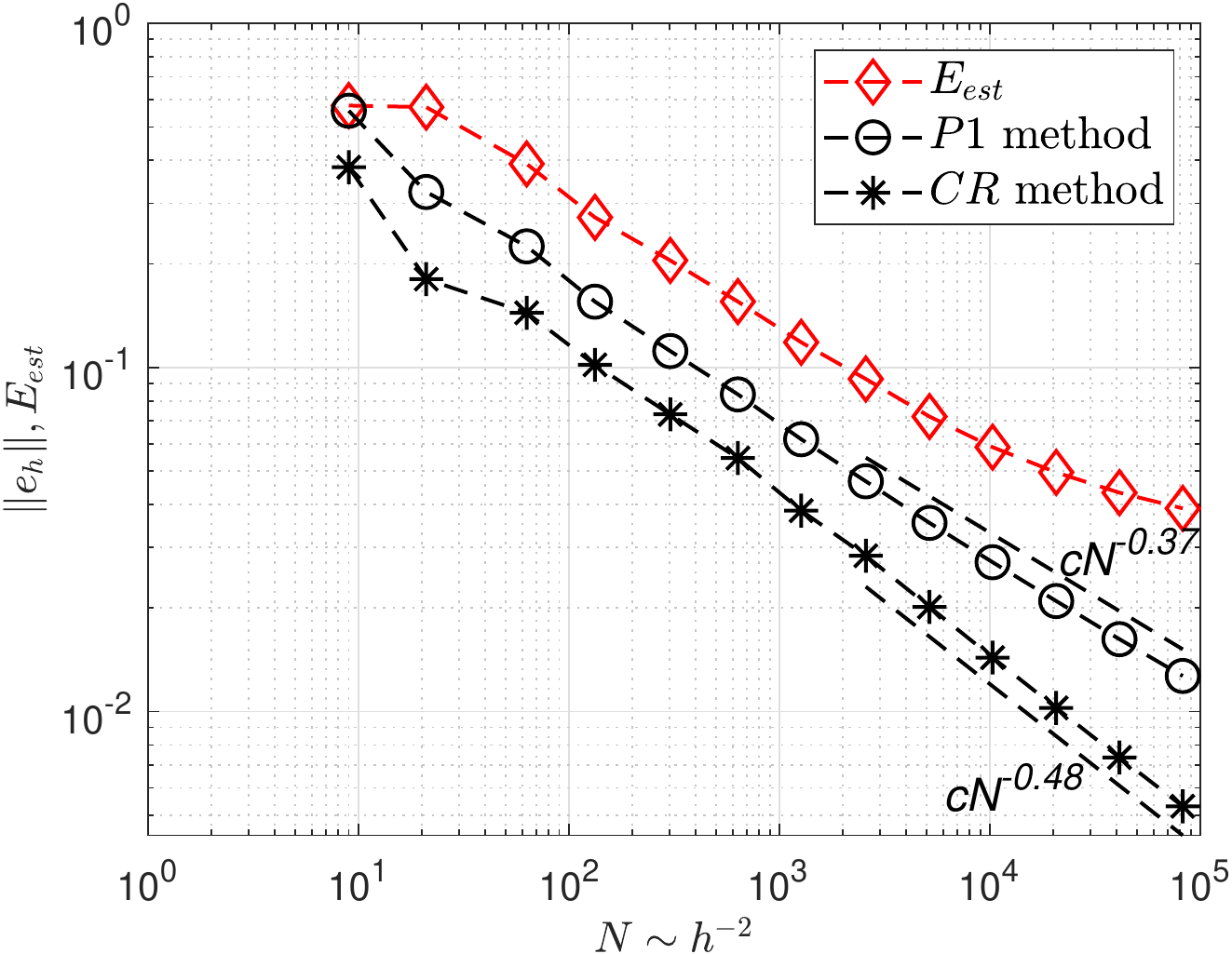}
\caption{\label{fig: conv graded 2d} Experimental convergence rates 
on quadratically graded triangulations 
in the approximation of the two-dimensional ROF-problem with 
piecewise constant solution specified in Example~\ref{graded one disc}.
Crouzeix--Raviart approximations lead to nearly linear convergence
of the $L^2$ error.}

\includegraphics[width = 9cm]{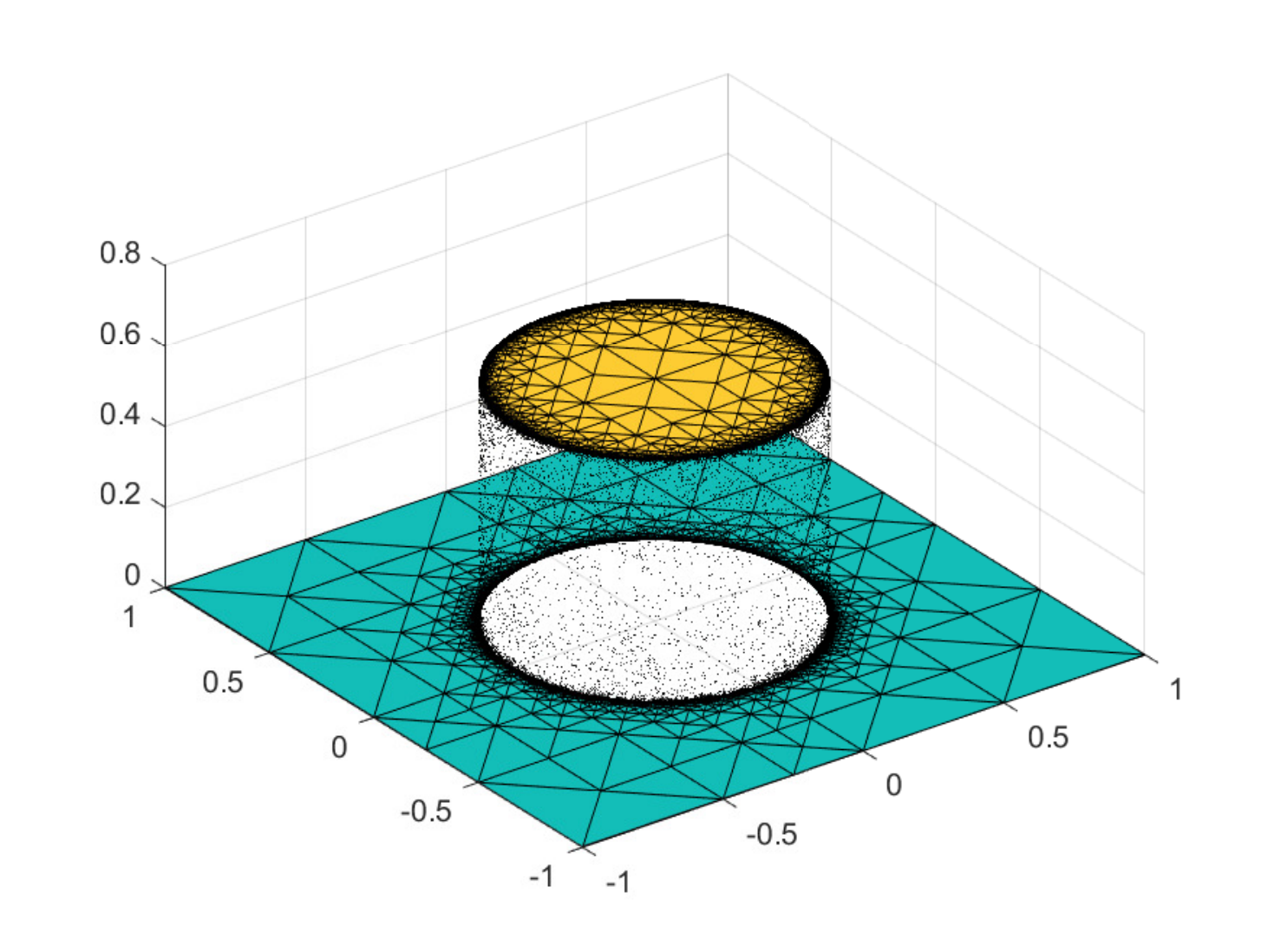}
\caption{\label{fig:graded_grid_2d} 
Projection $\Pi_h u_h$ of the Crouzeix--Raviart approximation on
the quadratically graded triangulation $\cT_{13}$ in Example~\ref{graded one disc}.
The localized refinement of the jump set leads to a high accuracy.}
\end{figure}

\subsection{Adaptive mesh refinement}
We finally investigate the automatic generation of locally refined
triangulations based on the a~posteriori error estimate provided
by Proposition~\ref{numerical_estimator}. We use the reconstructed, 
unscaled approximation~$z_h$ of the dual problem provided by the Crouzeix--Raviart
approximation $u_h^{cr}$ for the primal problem. This defines the
error estimator 
\[
E_{est} = \Big(\frac{2}{\a} \sum_{T\in \cT_h} \eta_{h,T}^2(u_h,z_h)\Big)^{1/2} 
+ \Big(\sum_{T\in\cT_h}  \|g - \Pi_h g\|_{L^2(T)}^2\Big)^{1/2},
\]
where the second sum contains data oscillation terms and the
first one the local refinement indicators $\eta_{h,T}$ which
are given by the element residuals
\[
\eta_{h,T}^2(u_h,z_h) = 
\int_T |\nabla u_h| - \nabla u_h \cdot \Pi_h z_h \dv{x}
+ \frac{1}{2\a} \int_T \big(\diver z_h - \a (u_h -g_h) \big)^2 \dv{x}.
\]
We follow established strategies in adaptive mesh refinement 
methods and select a minimal subset $M_h \subset \cT_h$ that constitutes
50\% of the total error estimator. We used the regularization parameter
$\veps = h^2$ to allow for an overall linear convergence rate, as
stopping criterion we used $\veps_{stop} = h^2/20$. 
We again use a setting that leads to a piecewise constant solution. 

\begin{example}[Piecewise constant solution] \label{graded one disc_b}
Let $\O = (-1,1)^2$, $\a=10$, and  $g = \chi_{B_r(0)}$ for $r = 1/2$. 
For homogeneous Dirichlet boundary conditions the minimizer of
the ROF model is given by $u=c_{r,\a} g$ with $c_{r,\a} = 1-2/(r\a)$. 
\end{example}

The experimental convergence rates for $P1$ and Crouzeix--Raviart
finite element approximations on adaptively generated triangulations
in Example~\ref{graded one disc_b} are shown in Figure~\ref{fig:conv_prim-dual}.
For both methods we observe an improvement over the optimal rate $O(h^{1/2})$
on sequences of uniform triangulations. The Crouzeix--Raviart method
leads to the experimental convergence rate $O(h^{0.76})$ while for
the $P1$ method we obtain the lower rate $O(h^{0.58})$.  
Our explanation for this is the good 
compatibilty of the Crouzeix--Raviart method specified
by the projection property of the quasi-interpolation operator
and the resulting discrete total-variation diminishing property. 
As addressed in Remark~\ref{rem:stronger_quant} the error estimator
$E_{est}$ cannot be expected to lead to meshes that are optimal for the
$L^2$ approximation error. 
The error estimator $E_{est}$ converges with nearly the same rate 
as the $P1$ approximation error indicating good reliability and
efficiency properties. A $P1$ finite element approximation obtained with the adaptive
mesh refinement strategy is shown in Figure~\ref {fig: approx_prim-dual_pic}
We observe an automatic local mesh refinement towards the discontinuity 
set of the solution but a weaker grading of approximately $\b \approx 1.7$
in comparison with 
Figure~\ref{fig:graded_grid_2d}. The reliable estimators
$\widehat{E}_{est}$ and $\widetilde{E}_{est}$, obtained from using the globally
and locally scaled vector fields $\hz_h$ and $\tz_h$ lead to meshes
on which these estimators converges suboptimally, cf. Figure~\ref{fig:conv_prim-dual}.
The $L^2$ error converged with similar rates reported above for meshes
constructed with $\widehat{E}_{est}$ but not with $\widetilde{E}_{est}$.

\begin{figure}[p]
\includegraphics[width=10cm]{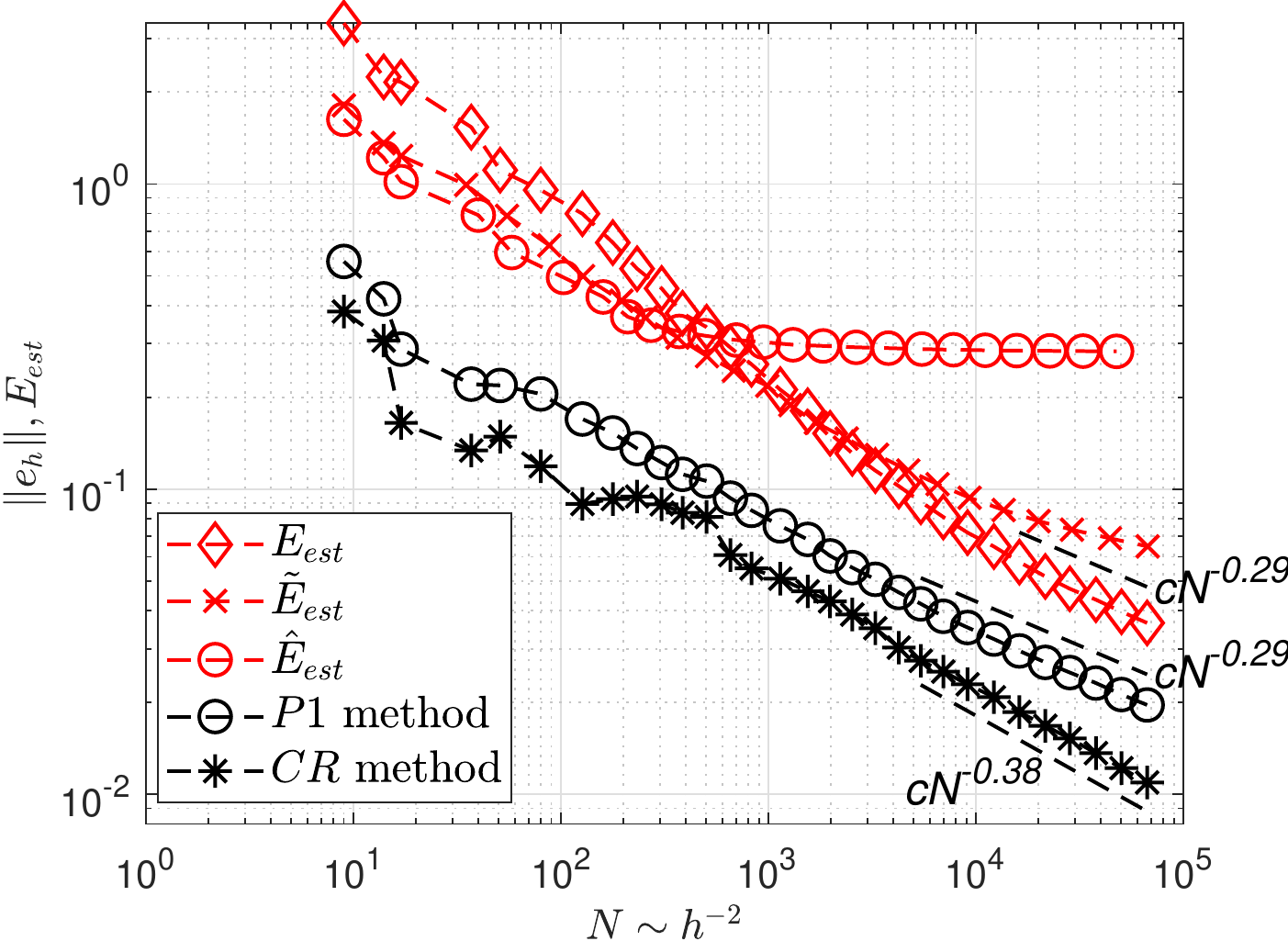} 
\caption{\label{fig:conv_prim-dual} 
Experimental convergence rates in the adaptive approximation of the ROF-problem 
defined in Example~\ref{graded one disc_b} using the primal-dual-gap error 
estimator $E_{est}$. Different experimental convergence rates are observed for
Crouzeix--Raviart and $P1$ finite element approximations. The estimators
$\widehat{E}_{est}$ and $\widetilde{E}_{est}$ obtained from globally and locally
scaled dual variables are inefficient.} 

\includegraphics[width=9cm]{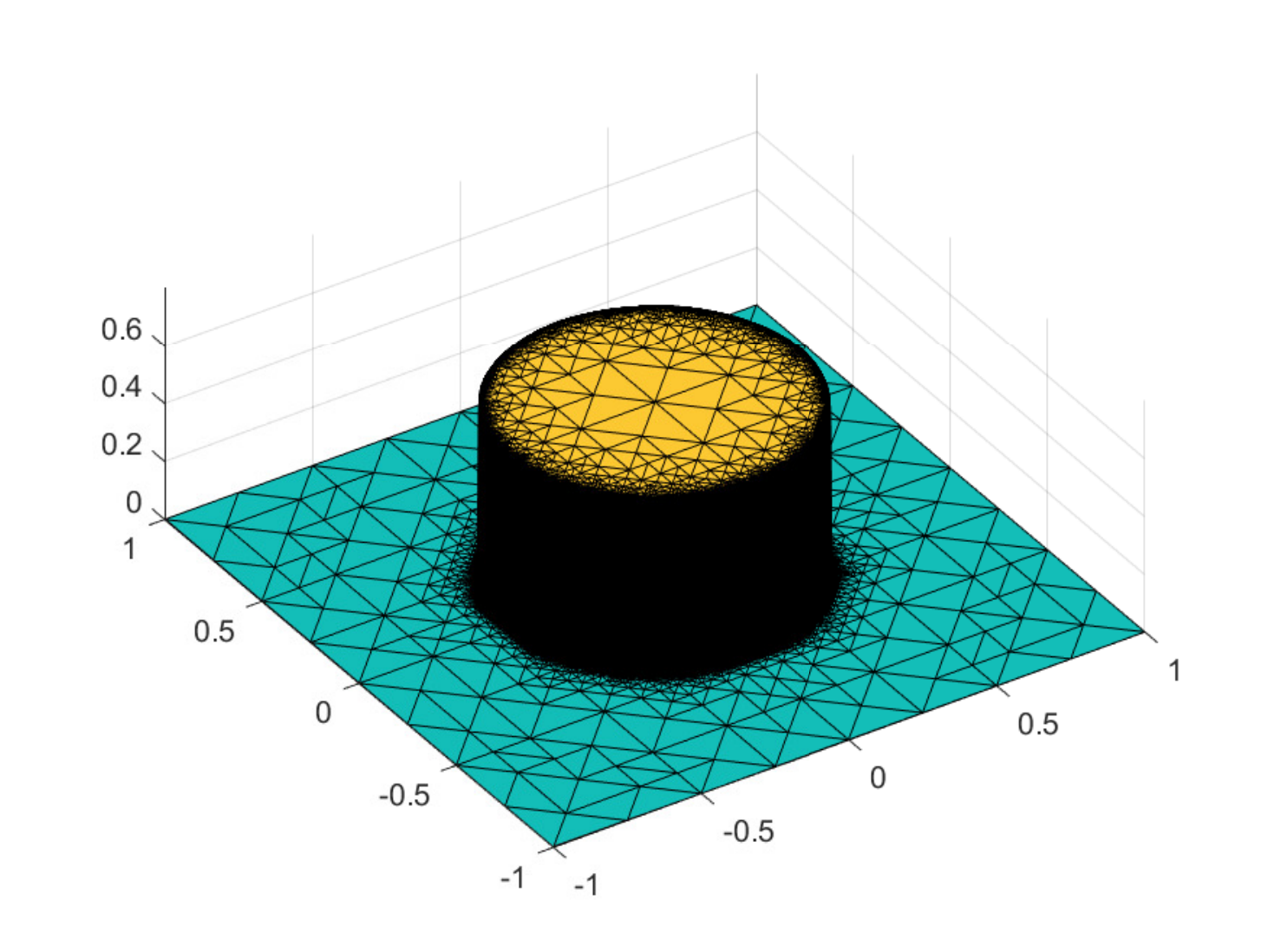}
\caption{\label{fig: approx_prim-dual_pic} 
Adaptively generated $P1$ approximation 
in Example~\ref{graded one disc_b}. The automatic mesh refinement
procedure leads to a local refinement in a neighborhood of the
discontinuity set.}
\end{figure}

\subsection*{Acknowledgments}
The authors thank Ricardo H. Nochetto for stimulating
discussions on various aspects of the results of this article. 
The first author acknowledges support by the DFG via the priority programme
SPP 1962 {\em Non-smooth and Complementarity-based Distributed Parameter Systems: 
Simulation and Hierarchical Optimization}. The second author acknowledges support 
from the ANR CIPRESSI project grant ANR-19-CE48-0017-01 of the French Agence Nationale 
de la Recherche.

\clearpage
 
\section*{References}
\printbibliography[heading=none]

@book {BreSco08-book,
    AUTHOR = {Brenner, Susanne C. and Scott, L. Ridgway},
     TITLE = {The mathematical theory of finite element methods},
    SERIES = {Texts in Applied Mathematics},
    VOLUME = {15},
   EDITION = {Third},
 PUBLISHER = {Springer, New York},
      YEAR = {2008},
     PAGES = {xviii+397},
      ISBN = {978-0-387-75933-3},
   MRCLASS = {65-01 (65-02)},
  MRNUMBER = {2373954},
       DOI = {10.1007/978-0-387-75934-0},
       URL = {https://doi.org/10.1007/978-0-387-75934-0},
}

@incollection {CCCNP10,
    AUTHOR = {Chambolle, Antonin and Caselles, Vicent and Cremers, Daniel
              and Novaga, Matteo and Pock, Thomas},
     TITLE = {An introduction to total variation for image analysis},
 BOOKTITLE = {Theoretical foundations and numerical methods for sparse
              recovery},
    SERIES = {Radon Ser. Comput. Appl. Math.},
    VOLUME = {9},
     PAGES = {263--340},
 PUBLISHER = {Walter de Gruyter, Berlin},
      YEAR = {2010},
   MRCLASS = {65D18 (26B30 68U10)},
  MRNUMBER = {2731599},
       DOI = {10.1515/9783110226157.263},
       URL = {https://doi.org/10.1515/9783110226157.263},
}

@BOOK{Bartels2016,
	AUTHOR = {Bartels, S{\"o}ren},
	YEAR = {2016},
	TITLE = {{Numerical Approximation of Partial Differential Equations} },
  	SERIES = {Texts in Applied Mathematics},
  	VOLUME = {64},
	EDITION = {},
	ISBN = {978-3-319-32354-1},
	PUBLISHER = {Springer},
	ADDRESS = {Berlin, Heidelberg},
}

@book {Bartels2015,
    AUTHOR = {Bartels, S\"{o}ren},
     TITLE = {Numerical methods for nonlinear partial differential
              equations},
    SERIES = {Springer Series in Computational Mathematics},
    VOLUME = {47},
 PUBLISHER = {Springer, Cham},
      YEAR = {2015},
     PAGES = {x+393},
      ISBN = {978-3-319-13796-4; 978-3-319-13797-1},
   MRCLASS = {65-01 (35A15 35A35 65Mxx 65Nxx)},
  MRNUMBER = {3309171},
MRREVIEWER = {Karsten Urban},
       DOI = {10.1007/978-3-319-13797-1},
       URL = {https://doi.org/10.1007/978-3-319-13797-1},
}

@book {Ambrosio2000,
    AUTHOR = {Ambrosio, Luigi and Fusco, Nicola and Pallara, Diego},
     TITLE = {Functions of bounded variation and free discontinuity
              problems},
    SERIES = {Oxford Mathematical Monographs},
 PUBLISHER = {The Clarendon Press, Oxford University Press, New York},
      YEAR = {2000},
     PAGES = {xviii+434},
      ISBN = {0-19-850245-1},
   MRCLASS = {49-02 (49J45 49K10 49Qxx)},
  MRNUMBER = {1857292},
MRREVIEWER = {J. E. Brothers},
}

@BOOK{Attouch2006,
	AUTHOR = {Attouch, Hedy AND Buttazzo, Giuseppe AND Michaille, Gérard},
	YEAR = {2006},
	TITLE = {{Variational Analysis in Sobolev and BV Spaces - Applications to PDEs and Optimization}},
  	SERIES = {MPS-SIAM Series on Optimization},
  	VOLUME = {4},
	EDITION = {},
	ISBN = {978-0-898-71600-9},
	PUBLISHER = {Society for Industrial and Applied Mathematics},
	ADDRESS = {Philadelphia},
}

@article{Veser,
author = {Fierro, Francesca and Veeser, Andreas},
title = {{A Posteriori Error Estimators for Regularized Total Variation of Characteristic Functions}},
journal = {SIAM Journal on Numerical Analysis},
volume = {41},
number = {6},
pages = {2032-2055},
year = {2003},
doi = {10.1137/S0036142902408283},
URL = {https://doi.org/10.1137/S0036142902408283},
}

@misc{Bart20b-pre,
    title={Error estimates for a class of discontinuous {G}alerkin methods for nonsmooth problems via convex duality relations},
    author={Sören Bartels},
    year={2020},
    eprint={2004.09196},
    archivePrefix={arXiv},
    primaryClass={math.NA}
}

@article{Chamb,
    AUTHOR = {Chambolle,Antonin AND Pock, Thomas},
    TITLE = {{Crouzeix-Raviart Approximation of the Total Variation on Simplicial Meshes}},
   JOURNAL = {Journal of Mathematical Imaging and Vision},
VOLUME = {62},
   YEAR = {2020}, 
PAGES = {872-899},
   URL = {https://doi.org/10.1007/s10851-019-00939-3},
}

@InProceedings{Raviart-Thomas,
author="Raviart, P. A.
and Thomas, J. M.",
editor="Galligani, Ilio
and Magenes, Enrico",
title="A mixed finite element method for 2-nd order elliptic problems",
booktitle="Mathematical Aspects of Finite Element Methods",
year="1977",
publisher="Springer Berlin Heidelberg",
address="Berlin, Heidelberg",
pages="292--315",
isbn="978-3-540-37158-8"
}

@article{40years-CR,
author = {Brenner, Susanne C.},
title = {{Forty Years of the Crouzeix-Raviart element}},
journal = {Numerical Methods for Partial Differential Equations},
volume = {31},
number = {2},
pages = {367-396},
doi = {https://doi.org/10.1002/num.21892},
url = {https://www.onlinelibrary.wiley.com/doi/abs/10.1002/num.21892},
year = {2015}
}

@article{Crouzeix-Raviart,
author = {Crouzeix, M. AND Raviart,  P.-A. },
title = {{Conforming and nonconforming finite element methods for solving the stationary Stokes equations I}},
journal = { 	
R.A.I.R.O. },
volume = {7},
number = {R3},
pages = {33-75},
year = {1973},
doi = {https://doi.org/10.1051/m2an/197307R300331 },
}

@article{Jump-set-paper,
author = {Caselles, Vicent and Chambolle, Antonin and Novaga, Matteo},
title = {{The Discontinuity Set of Solutions of the TV Denoising Problem and Some Extensions}},
journal = {Multiscale Modeling \& Simulation},
volume = {6},
number = {3},
pages = {879-894},
year = {2007},
doi = {10.1137/070683003},
URL = {https://doi.org/10.1137/070683003},
}

@article{BaDiNo18,
  author = {Bartels, S\"{o}ren and Diening, Lars and Nochetto, Ricardo H.},
  title = {{Unconditional stability of semi-implicit discretizations of
              singular flows}},
  journal = {SIAM J. Numer. Anal.},
  fjournal = {SIAM Journal on Numerical Analysis},
  volume = {56},
  year = {2018},
  number = {3},
  pages = {1896--1914},
  issn = {0036-1429},
  mrclass = {65M60 (35J20 35K92 65M12 65M15)},
  mrnumber = {3819161},
  doi = {10.1137/17M1159166},
  url = {https://doi.org/10.1137/17M1159166}
}

@article{BaNoSa15,
  author = {Bartels, S{\"o}ren and Nochetto, Ricardo H. and Salgado, Abner
              J.},
  title = {A total variation diminishing interpolation operator and
              applications},
  journal = {Math. Comp.},
  fjournal = {Mathematics of Computation},
  volume = {84},
  year = {2015},
  number = {296},
  pages = {2569--2587},
  issn = {0025-5718},
  mrclass = {65N30 (41A05 49J40 49M25 65D05 65N15)},
  mrnumber = {3378839},
  mrreviewer = {Christian Wieners},
  doi = {10.1090/mcom/2942},
  url = {http://dx.doi.org/10.1090/mcom/2942}
}

@article{ROF,
title = "Nonlinear total variation based noise removal algorithms",
journal = "Physica D: Nonlinear Phenomena",
volume = "60",
number = "1",
pages = "259 - 268",
year = "1992",
issn = "0167-2789",
doi = "https://doi.org/10.1016/0167-2789(92)90242-F",
url = "http://www.sciencedirect.com/science/article/pii/016727899290242F",
author = "Leonid I. Rudin and Stanley Osher and Emad Fatemi"
}

@incollection {Burg16,
    AUTHOR = {Burger, Martin},
     TITLE = {Bregman distances in inverse problems and partial differential
              equations},
 BOOKTITLE = {Advances in mathematical modeling, optimization and optimal
              control},
    SERIES = {Springer Optim. Appl.},
    VOLUME = {109},
     PAGES = {3--33},
 PUBLISHER = {Springer, [Cham]},
      YEAR = {2016},
   MRCLASS = {65J22 (90C25 90C46)},
  MRNUMBER = {3526561},
       DOI = {10.1007/978-3-319-30785-5\_2},
       URL = {https://doi.org/10.1007/978-3-319-30785-5_2},
}

@article {HHSVW19,
    AUTHOR = {Herrmann, Marc and Herzog, Roland and Schmidt, Stephan and
              Vidal-N\'{u}\~{n}ez, Jos\'{e} and Wachsmuth, Gerd},
     TITLE = {Discrete total variation with finite elements and applications
              to imaging},
   JOURNAL = {J. Math. Imaging Vision},
  FJOURNAL = {Journal of Mathematical Imaging and Vision},
    VOLUME = {61},
      YEAR = {2019},
    NUMBER = {4},
     PAGES = {411--431},
      ISSN = {0924-9907},
   MRCLASS = {94A08 (49M29 65K05 65N30 68U10)},
  MRNUMBER = {3937961},
MRREVIEWER = {Andreas Langer},
       DOI = {10.1007/s10851-018-0852-7},
       URL = {https://doi.org/10.1007/s10851-018-0852-7},
}

@article {Bart12,
    AUTHOR = {Bartels, S\"{o}ren},
     TITLE = {Total variation minimization with finite elements: convergence
              and iterative solution},
   JOURNAL = {SIAM J. Numer. Anal.},
  FJOURNAL = {SIAM Journal on Numerical Analysis},
    VOLUME = {50},
      YEAR = {2012},
    NUMBER = {3},
     PAGES = {1162--1180},
      ISSN = {0036-1429},
   MRCLASS = {65K10 (35A15 68U10 94A08)},
  MRNUMBER = {2970738},
MRREVIEWER = {H. P. Dikshit},
       DOI = {10.1137/11083277X},
       URL = {https://doi.org/10.1137/11083277X},
}

@article {BeEfRu17,
    AUTHOR = {Berkels, Benjamin and Effland, Alexander and Rumpf, Martin},
     TITLE = {A posteriori error control for the binary {M}umford-{S}hah
              model},
   JOURNAL = {Math. Comp.},
  FJOURNAL = {Mathematics of Computation},
    VOLUME = {86},
      YEAR = {2017},
    NUMBER = {306},
     PAGES = {1769--1791},
      ISSN = {0025-5718},
   MRCLASS = {94A08 (49M25 53C22 65D18 65K10)},
  MRNUMBER = {3626536},
MRREVIEWER = {Marius Ghergu},
       DOI = {10.1090/mcom/3138},
       URL = {https://doi.org/10.1090/mcom/3138},
}

@article {Bart15,
    AUTHOR = {Bartels, S\"{o}ren},
     TITLE = {Error control and adaptivity for a variational model problem
              defined on functions of bounded variation},
   JOURNAL = {Math. Comp.},
  FJOURNAL = {Mathematics of Computation},
    VOLUME = {84},
      YEAR = {2015},
    NUMBER = {293},
     PAGES = {1217--1240},
      ISSN = {0025-5718},
   MRCLASS = {65N30 (65N15 65N50)},
  MRNUMBER = {3315506},
       DOI = {10.1090/S0025-5718-2014-02893-7},
       URL = {https://doi.org/10.1090/S0025-5718-2014-02893-7},
}

@article {HinKun04,
    AUTHOR = {Hinterm\"{u}ller, M. and Kunisch, K.},
     TITLE = {Total bounded variation regularization as a bilaterally
              constrained optimization problem},
   JOURNAL = {SIAM J. Appl. Math.},
  FJOURNAL = {SIAM Journal on Applied Mathematics},
    VOLUME = {64},
      YEAR = {2004},
    NUMBER = {4},
     PAGES = {1311--1333},
      ISSN = {0036-1399},
   MRCLASS = {49M29 (49K10 49K30 65K10 94A08)},
  MRNUMBER = {2068672},
MRREVIEWER = {A. Bultheel},
       DOI = {10.1137/S0036139903422784},
       URL = {https://doi.org/10.1137/S0036139903422784},
}

@article {ChaLio97,
    AUTHOR = {Chambolle, Antonin and Lions, Pierre-Louis},
     TITLE = {Image recovery via total variation minimization and related
              problems},
   JOURNAL = {Numer. Math.},
  FJOURNAL = {Numerische Mathematik},
    VOLUME = {76},
      YEAR = {1997},
    NUMBER = {2},
     PAGES = {167--188},
      ISSN = {0029-599X},
   MRCLASS = {65K10 (49J99 68U10)},
  MRNUMBER = {1440119},
MRREVIEWER = {Tom\'{a}\v{s} Roub\'{\i}\v{c}ek},
       DOI = {10.1007/s002110050258},
       URL = {https://doi.org/10.1007/s002110050258},
}

@incollection{ChaPoc21,
  TITLE = {{Approximating the Total Variation with Finite Differences or Finite Elements}},
  AUTHOR = {Chambolle, Antonin and Pock, Thomas},
  URL = {https://hal.archives-ouvertes.fr/hal-02959358},
  NOTE = {{\`a} paraitre},
  BOOKTITLE = {{Handbook of Numerical Analysis: Geometric Partial Differential Equations II}},
  SERIES = {Handbook of Numerical Analysis: Geometric Partial Differential Equations II},
  YEAR = {2021},
  HAL_ID = {hal-02959358},
  HAL_VERSION = {v2},
}

@article {ChLeLu11,
    AUTHOR = {Chambolle, Antonin and Levine, Stacey E. and Lucier, Bradley
              J.},
     TITLE = {An upwind finite-difference method for total variation-based
              image smoothing},
   JOURNAL = {SIAM J. Imaging Sci.},
  FJOURNAL = {SIAM Journal on Imaging Sciences},
    VOLUME = {4},
      YEAR = {2011},
    NUMBER = {1},
     PAGES = {277--299},
   MRCLASS = {65D18 (68U10 94A08)},
  MRNUMBER = {2792413},
MRREVIEWER = {Dragos Calitoiu},
       DOI = {10.1137/090752754},
       URL = {https://doi.org/10.1137/090752754},
}

@article {LaiMat12,
    AUTHOR = {Lai, Ming-Jun and Matamba Messi, Leopold},
     TITLE = {Piecewise linear approximation of the continuous
              {R}udin-{O}sher-{F}atemi model for image denoising},
   JOURNAL = {SIAM J. Numer. Anal.},
  FJOURNAL = {SIAM Journal on Numerical Analysis},
    VOLUME = {50},
      YEAR = {2012},
    NUMBER = {5},
     PAGES = {2446--2466},
      ISSN = {0036-1429},
   MRCLASS = {94A08 (65K10)},
  MRNUMBER = {3022226},
MRREVIEWER = {Luis G\'{o}mez},
       DOI = {10.1137/110854539},
       URL = {https://doi.org/10.1137/110854539},
}

@article {WanLuc11,
    AUTHOR = {Wang, Jingyue and Lucier, Bradley J.},
     TITLE = {Error bounds for finite-difference methods for
              {R}udin-{O}sher-{F}atemi image smoothing},
   JOURNAL = {SIAM J. Numer. Anal.},
  FJOURNAL = {SIAM Journal on Numerical Analysis},
    VOLUME = {49},
      YEAR = {2011},
    NUMBER = {2},
     PAGES = {845--868},
      ISSN = {0036-1429},
   MRCLASS = {65N06 (65D18 65N15 94A08)},
  MRNUMBER = {2792398},
MRREVIEWER = {Shai Dekel},
       DOI = {10.1137/090769594},
       URL = {https://doi.org/10.1137/090769594},
}

@unpublished{CaiCha20-pre,
  TITLE = {{Error estimates for finite differences approximations of the total variation}},
  AUTHOR = {Caillaud, Corentin and Chambolle, Antonin},
  URL = {https://hal.archives-ouvertes.fr/hal-02539136},
  NOTE = {HAL preprint nr. 02539136},
  YEAR = {2020},
  MONTH = Apr,
  HAL_ID = {hal-02539136},
  HAL_VERSION = {v1},
}

@article {BarMil20,
    AUTHOR = {Bartels, S\"{o}ren and Milicevic, Marijo},
     TITLE = {Primal-dual gap estimators for {\it a posteriori} error
              analysis of nonsmooth minimization problems},
   JOURNAL = {ESAIM Math. Model. Numer. Anal.},
  FJOURNAL = {ESAIM. Mathematical Modelling and Numerical Analysis},
    VOLUME = {54},
      YEAR = {2020},
    NUMBER = {5},
     PAGES = {1635--1660},
      ISSN = {0764-583X},
   MRCLASS = {49M29 (49J10 65K15 65N15 65N50 90C25 90C48)},
  MRNUMBER = {4127951},
       DOI = {10.1051/m2an/2019074},
       URL = {https://doi.org/10.1051/m2an/2019074},
}

@phdthesis{tovey2020,
  title={Mathematical Challenges in Electron Microscopy},
  author={Tovey, Robert},
  year={2020},
  school={University of Cambridge},
  doi={10.17863/CAM.63763}
}

@article{Bart20a,
title = {Nonconforming discretizations of convex minimization problems and precise relations to mixed methods},
journal = {Computers \& Mathematics with Applications},
volume = {93},
pages = {214-229},
year = {2021},
issn = {0898-1221},
doi = {https://doi.org/10.1016/j.camwa.2021.04.014},
url = {https://www.sciencedirect.com/science/article/pii/S0898122121001541},
author = {Sören Bartels},
}
\end{document}